\documentclass[12pt, reqno, a4paper]{amsart}
\usepackage[all]{xy}

\usepackage{ amssymb, amsmath, enumerate, amsfonts, amsthm, mathrsfs, url, bm, mathtools, comment}

\numberwithin{equation}{section}

\usepackage{xcolor}  	
\usepackage[backref]{hyperref}
\hypersetup{
	colorlinks,
    linkcolor={blue!60!black},
    citecolor={blue!60!black},
    urlcolor={red!60!black}
}

\usepackage{color}

\usepackage[margin=1.25in]{geometry}

\RequirePackage{doi}

\usepackage[square,sort,comma,numbers]{natbib}
\makeatletter
\@namedef{subjclassname@2020}{%
  \textup{2020} Mathematics Subject Classification}
\makeatother

 \newcommand{\set}[1]{\left\{#1\right\}}
\newcommand{\bigset}[1]{\bigl\{ #1 \bigr\}}

\newcommand{\abs}[1]{\left| #1\right|}
\newcommand{\bigabs}[1]{\bigl| #1 \bigr|}
\newcommand{\Bigabs}[1]{\Bigl| #1 \Bigr|}
\newcommand{\biggabs}[1]{\biggl| #1 \biggr|}
\newcommand{\Biggabs}[1]{\Biggl| #1 \Biggr|}

\newcommand{\brac}[1]{\left( #1 \right)}
\newcommand{\bigbrac}[1]{\bigl( #1 \bigr)}
\newcommand{\Bigbrac}[1]{\Bigl( #1 \Bigr)}
\newcommand{\biggbrac}[1]{\biggl( #1 \biggr)}
\newcommand{\Biggbrac}[1]{\Biggl( #1 \Biggr)}
\newcommand{\norm}[1]{\left\| #1\right\|}

\newcommand{\ang}[1]{\left\langle#1\right\rangle}

\newcommand{\rd}{\,\mathrm{d}}

\newcommand{\twosum}[2]{ \sum_{\substack{#1\\ #2}}}
\newcommand{\threesum}[3]{ \sum_{\substack{#1\\ #2\\ #3}}}

\newcommand{\N}{\mathbb{N}}
\newcommand{\Z}{\mathbb{Z}}

\newcommand{\R}{\mathbb{R}}
\newcommand{\C}{\mathbb{C}}
\newcommand{\T}{\mathbb{T}}

\newcommand{\E}{\mathbb{E}}

\newcommand{\sgn}{\mathrm{sgn}}

\newcommand{\meas}{\mathrm{meas}}

\newcommand{\supp}{\mathrm{supp}}

\newcommand{\eps}{\varepsilon}

\newcommand{\CB}{\mathcal{B}}

\newcommand{\poly}{\mathrm{poly}}
\newcommand{\lip}{\mathrm{Lip}}

\renewcommand{\mod}[1]{\,(\mathrm{mod}\,#1)}
\renewcommand{\bar}{\overline}

\renewcommand{\leq}{\leqslant}
\renewcommand{\geq}{\geqslant}
\renewcommand{\epsilon}{\varepsilon}

\newtheorem{theorem}{Theorem}[section]

\newtheorem{proposition}[theorem]{Proposition}
\newtheorem{lemma}[theorem]{Lemma}

\theoremstyle{definition}
\newtheorem{definition}[theorem]{Definition}
\newtheorem*{remark}{Remark}

\numberwithin{theorem}{section}

\author{Joni Ter\"av\"ainen}
\address{Department of Pure Mathematics and Mathematical Statistics,  University of Cambridge, CB3 0WB, UK}
\email{joni.p.teravainen@gmail.com}

\author{Mengdi Wang}
\address{\'{E}cole Polytechnique F\'{e}d\'{e}rale de Lausanne (EPFL), Lausanne, Switzerland}
\email{mengdi.wang@epfl.ch}

\begin{document}

\title{On the Green--Tao theorem for sparse sets}

\begin{abstract} We establish the following quantitative form of the Green--Tao theorem: if a set $\mathcal{A}$ of relative density $\delta$ within the primes up to $N$ contains no nontrivial arithmetic progressions of length $k\geq 4$, then $\delta\ll \exp(-(\log \log \log N)^{c_k})$ for some $c_k>0$. This improves on previous work of Rimani\'c and Wolf.

The main new ingredients in the proof are a version of the Leng--Sah--Sawhney quasipolynomial inverse theorem for unbounded functions and a dense model theorem with quasipolynomial dependencies, which may be of independent interest. 

\end{abstract}

\maketitle

\section{Introduction}

About two decades ago, Green and Tao~\cite{GT-prime} achieved a landmark result by proving that the primes contain arbitrarily long arithmetic progressions. This theorem can be viewed as a counterpart to Szemer\'edi's theorem on arithmetic progressions in the integers, and it naturally raises the question of obtaining quantitative bounds in the setting of primes. For 3-term arithmetic progressions (3-APs), Green's original work~\cite{G-trans} established that if a subset of\footnote{Here and throughout, $\mathbb{P}$ denotes the set of primes.} $\mathbb{P}\cap [N]$ contains no nontrivial 3-APs, then it has relative density $O((\log \log \log \log N)^{-c})$ for some $c>0$, which was later improved to a double-logarithmic decay by Helfgott--de Roton~\cite{HR}, and the power of the double logarithm was subsequently improved by Naslund~\cite{Nas}.

Before turning to the case of primes in detail, let us first recall the corresponding developments in the integers. A recent breakthrough of Kelley and Meka~\cite{KM} and its improvement by Bloom and Sisask~\cite{bloom-sisask} show that any subset $\mathcal A$ of $[N]\coloneqq  \set{1,2,\dots,N}$ that contains no nontrivial 3-APs must satisfy the density bound 
\[
\frac{|\mathcal A|}{N}\ll \exp\bigbrac{-c(\log N)^{1/9}}
\]
for some constant $c>0$. However, the Fourier-analytic methods used for these results become inadequate for progressions of length $k\geq 4$ and less is known about these higher-order cases. The current best bounds for sets $\mathcal A$ of $[N]$ lacking $k$-APs are
\begin{align}\label{eq:sz}
\frac{|\mathcal A|}{N}\ll\begin{cases}(\log N)^{-c_4},\quad &k=4,\\ \exp(-(\log \log N)^{c_k}),\quad &k\geq 5,\end{cases}    
\end{align}
due to Green--Tao~\cite{GT-poly-4aps} for $k=4$ and Leng--Sah--Sawhney~\cite{LSS2} for $k\geq 5$, improving on work of Gowers~\cite{Go}.

 In the case of polynomial progressions, Shao and the second author~\cite{SW} proved a logarithmic bound for the density of sets lacking configurations of the form $x,x+P_1(d),\dots,x+P_k(d)$ with $d\neq 0$, provided that the polynomials $P_1,\dots,P_k\in \Z[y]$ have distinct degrees,  thereby improving on an earlier result of Peluse~\cite{Pe}. See~\cite{Pre14},~\cite{PPS},~\cite{kuca},~\cite{peluse-sah-sawhney},~\cite{altman-sawhney} for some other results in the setting of polynomial progressions.

Since the primes have density $O((\log N)^{-1})$ within $[N]$ for large $N$, Bloom and Sisask's bound~\cite{bloom-sisask} immediately implies that any subset of primes up to $N$ with relative density $\exp\bigbrac{-c'(\log N)^{1/9}}$ for some constant $c'>0$ must contain nontrivial 3-APs. Nonetheless, the current quantitative versions of  Szemer\'edi's theorem remain far from establishing the existence of $k$-APs  in the primes when $k\geq 4$. It therefore remains important to investigate the quantitative bounds and the underlying techniques in the primes (and other sparse sets). In this direction, Rimani\'c and Wolf~\cite{RW} established the following result: if a set $\mathcal A\subseteq[N]\cap \mathbb P$ does not contain nontrivial $k$-APs, then 
\[
\frac{|\mathcal A|}{|[N]\cap \mathbb P|} \ll 
\begin{cases}
	 (\log \log \log N)^{-c_4} \qquad &\text{ when }k=4,\\
	(\log \log \log \log N)^{-c_k}  & \text{ when } k\geq 5,
\end{cases}
\] 
for  some small constant $c_k>0$. By inputting improved quantitative bounds for Szemer\'edi's theorem into~\cite{RW}, one can remove one logarithm from these bounds in the case $k\geq 5$. 

In this paper, we prove the following exponential improvement.

\begin{theorem}[Density bound]\label{density-bound}
Let $k\geq 4$ be a natural number, and let $N\geq 100$. Let $\mathcal A\subseteq[N]\cap \mathbb P$ be a subset of primes. If $\mathcal A$ does not contain any nontrivial $k$-term arithmetic progression then for some $c_k>0$ we have
\[
\frac{|\mathcal A|}{|[N]\cap \mathbb P|} \ll \begin{cases}(\log \log N)^{-c_4}\quad &\text{ when }k=4,\\ \exp(-(\log \log \log N)^{c_k}),\quad &\text{ when }k\geq 5.\end{cases}
\]
 \end{theorem}

 \subsection{Proof outline} 

Our proof of Theorem~\ref{density-bound} proceeds through the framework of the \emph{transference principle}, originally introduced by Green~\cite{G-trans}, and refined in the works~\cite{GT-prime},~\cite{TZ-polynomial-sze},~\cite{reing},~\cite{Go-dense},~\cite{CFZ}, among others. The broader aim of this paper is to construct a \emph{dense model} with quasipolynomial bounds, formulated precisely in Lemma~\ref{tra-wea} and Proposition~\ref{tra-wea-2}. Roughly speaking, these results assert the following: if a function $f\colon \mathbb{Z}\to\mathbb{R}_{\geq 0}$ supported on $[N]$ admits a pseudorandom majorant $\nu\colon \mathbb{Z} \to \R_{\geq 0}$  satisfying the Gowers norm bound
\begin{align}\label{eq:nubound}
\norm{\nu -1}_{U^{2k}[N]} \leq \eps 
\end{align}
 for some small parameter $0<\eps<1$, and if, for technical reasons, we additionally assume the mild $L^{\infty}$ bound 
 \[
 \norm{\nu}_\infty \ll \exp\exp\bigbrac{(\log(1/\eps))^{c_k'}},
 \]
for some small constant $c_k'>0$,
 then there exists a bounded function $g\colon [N] \to [0,2]$ such that
\begin{align}\label{unif-intr}
\norm{f-g}_{U^k[N]} \ll_k \exp\bigbrac{- (\log(1/\eps))^{\gamma_k}}
\end{align}
for some small $\gamma_k>0$. 
This dense model allows one to transfer multilinear counting expressions from the (potentially unbounded) function $f$  to the bounded function $g$.\footnote{We note that our pseudorandomness hypothesis is weaker than in~\cite{GT-prime} or~\cite{CFZ}. The possibility of such a relaxation to a $U^{2k}[N]$-norm bound was first shown in the qualitative setting in~\cite{dodos}.} Once such a function $g$ is in place, the remaining arguments become essentially standard. We emphasise that~\eqref{unif-intr} improves the exponential bounds obtained in~\cite{Go-dense} and~\cite{reing}, and Theorem~\ref{density-bound} crucially relies on this strengthened estimate.

We briefly sketch the transference argument. The version of Szemer\'edi's theorem due to Varnavides~\cite{Var} asserts that for $N\geq N_0(\delta,k)$ large enough, for any $1$-bounded function $h\colon \mathbb{Z}\to [0,1]$ with average $\delta=\E_{n\in [N]} h(n)$, one has  
\[
N^{-1} \sum_{x\in \Z}\E_{d\in [N]} h(x) h(x+d) \cdots h(x+(k-1)d) \geq c(\delta,k)
\] 
for some function $c(\delta,k)>0$. Both functions $N_0(\delta,k)$ and $c(\delta,k)$ can be made explicit by using the best known quantitative bounds for Szemer\'edi's theorem.
On the other hand, the generalised von Neumann theorem (Lemma~\ref{neumann}) shows that the difference between the counting expressions weighted by $f$ and $g$ can be controlled by a suitable Gowers norm. Let $f$ be a suitable $W$-tricked version of the function $\Lambda\cdot 1_{\mathcal{A}}$, where $\Lambda$ is the von Mangoldt function. Assuming that we have a pseudorandom majorant $\nu$ as in~\eqref{eq:nubound} with $\varepsilon=(\log N)^{-c_k'}$ for some small $c_k'>0$, from the above we obtain a function $g\colon \Z\to [0,2]$ that is supported on the interval $[N]$ and satisfies~\eqref{unif-intr}, and then we can estimate
\begin{align*}
&N^{-1} \sum_{x\in \Z}\E_{d\in [N]} f(x) f(x+d) \cdots f(x+(k-1)d)\\
=&N^{-1} \sum_{x\in \Z} \E_{d\in [N]} g(x) g(x+d) \cdots g(x+(k-1)d) +O(\norm{f-g}_{U^{k-1}[N]})\\
\geq &2^kc(\delta/2,k) - O_k(\exp\bigbrac{- (\log(1/\eps))^{\gamma_k}})	
\end{align*}
for some $\gamma_k>0$.
To obtain a nontrivial bound, we therefore require
\[
c(\delta/2,k) \gg \exp\bigbrac{- (\log(1/\eps))^{\gamma_k/2}},
\]
and now the range of $\delta$ in which we obtain a contradiction follows from the quantitative Szemer\'edi theorem.
 We can indeed take $\varepsilon$ to be of size $(\log N)^{-c_k'}$ for some $c_k'>0$ by using the Goldston--Y{\i}ld{\i}r{\i}m type asymptotics of Green and Tao~\cite{GT-linear}, as in ~\cite{TT}. However, it appears very difficult to take $\eps$ significantly smaller than a small negative power of $\log N$, since the pseudorandom majorant $\nu$ is essentially a $W$-tricked Selberg sieve, with $W=\prod_{p\leq w}p$, and such a function is Gowers uniform only up to accuracy $w^{-c_k'}$ for some $c_k'>0$. Current technology seems to limit us to $W\ll N^{c}$ for some $c>0$, and hence we must restrict to $w$ somewhat smaller than $\log N$. 
Consequently, our argument suffers an additional logarithmic loss compared with the dense setting.

In the final part of this subsection, we briefly outline our approach to constructing a dense model. The  objective, as suggested by inequality~\eqref{unif-intr}, is to find a structured function that approximates $f$ in the Gowers uniformity norm. A trivial ansatz for such a structured function is the global average $\E_{[N]}(f)$, which is constant on the entire interval $[N]$. If $f$ is sufficiently pseudorandom, one would expect 
\[
\norm{f-\E_{[N]} (f)}_{U^k[N]} \leq\eps.
\]
 Otherwise, we  refine $[N]$ using a partition $\CB =\set{B_1,\dots,B_m}$ and consider the conditional expectation $\E_\CB (f)$, which is locally constant on each atom $B_i\in \CB$. We then test whether $\norm{f-\E_{\CB} (f)}_{U^k[N]} \leq\eps $ or not. 
 
 The inverse theorem (Theorem~\ref{tra-inv}) serves as the tool for generating successive refinements of the partition, while Pythagoras's theorem (Lemma~\ref{pythagoras}) allows us to accumulate the gain from each step. However, a remaining issue is to ensure that the iteration terminates with usable information. For a $1$-bounded function $g$, one enforces termination via the bound $\|g\|_2\leq 1$; this approach appears in, for instance,~\cite{LSS2,PP2,PPS,SW}. In our setting, however, $f$ is modelling the von Mangoldt function, and typically $\norm{f}_2^2\gg \log N$. Since the gain at each step is small (for the primes, we can only ensure a gain on the order of $\exp(-(\log\log N)^{1/100})$, say), we cannot afford such a long iteration. However, since the iteration produces a sequence of increasingly refined partitions, if it does not run for too long then we can guarantee that \emph{most} atoms are not too small, and averaging over these atoms helps flatten the peaks of $f$. Using this observation, we get around this issue by showing that $\norm{\E_\CB(f)}_2\ll 1$ for a large class of partitions $\CB$, which is sufficient to run the iteration.  As expected, $\E_\CB(f)$ will be constant on the atoms of the final partition, but it is not clear that it is $\ell^2$-norm bounded. To establish this, we show that each atom $B_i\in \CB$ can be approximated by suitable nilsequences, and then deduce boundedness by controlling the correlation of $f$ with these nilsequences by using the nilpotent Hardy--Littlewood method (see Proposition~\ref{cor-1}).


\subsection{Organisation of the paper}

Section~\ref{sec2} establishes a transferred inverse theorem for the Gowers uniformity norms with quasipolynomial bounds; this result is of independent interest and may have applications beyond the present setting. 

In Section~\ref{sec3}, our goal is to prove the dense model theorem with quasipolynomial dependencies. In Lemma~\ref{tra-wea}, we apply an energy increment argument together with the transferred inverse theorem to show that the target function $f$ can be approximated in $U^k$ norm by a conditional expectation of $f$ on a factor given in terms of level sets of nilsequences. We then further show that the indicator functions of the atoms of this factor can be well approximated by nilsequences. We end this section with Proposition~\ref{tra-wea-2}, which is a dense model theorem stating that if a certain technical condition relating to $f$ correlating with only ``major arc'' nilsequences holds, then the model of $f$ given by the conditional expectation is a bounded function.

In Section~\ref{sec4}, we use Vaughan's identity and Type I and II estimates for nilsequences to show that the technical nilsequence condition from Proposition~\ref{tra-wea-2} is satisfied in the case of the von Mangoldt function.

In Section~\ref{sec5} we prove a generalised von Neumann theorem with explicit quantitative dependencies suitable for our application.

Finally, in Section~\ref{sec6} we put these ingredients together to deduce Theorem~\ref{density-bound}.

\subsection{Notation}

We record here some basic notational conventions that will be used throughout the paper. If $A$ is a set, we use $1_A$ to denote the indicator of $A$; and if $A$ is a statement, we let $1_A$ equal to $1$ when $A$ is true and to $0$ when $A$ is false. Whenever $\theta \in \mathbb{R}$, we write $\|\theta\|_\T$ for the distance from $\theta$ to the nearest integer. When no confusion arises, we abbreviate this as $\norm{\theta}$.

We use the asymptotic notation $X\ll Y$, $Y\gg X$ or $X=O(Y)$ to indicate that $|X|\leq CY$ for some constant $C>0$. 

For a nonempty finite set $A$ and a function $f\colon A \to \C$ we use the averaging notation
\[
\E_{x\in A} f(x) =\frac{1}{|A|} \sum_{x\in A} f(x).
\]
For $1\leq p\leq \infty$ and functions $f,g\colon [N] \to\C$, we define the (normalised) $\ell^p$ norms by
\[
\norm{f}_{p} = \bigbrac{\E_{n\in [N]} |f(n)|^p}^{1/p}
\]
with the usual convention that $\norm{f}_\infty = \sup_{n\in [N]}|f(n)|$. Note that the definition of $\norm{f}_{p}$ depends on $N$, but this will always be clear from context. We also define the inner product 
\[
\ang{f,g} =\E_{n\in [N]} f(n) \overline{g(n)}.
\]

\subsection*{Gowers uniformity norm}
Let $G$ be an abelian group, let $h\in G$, and let $f\colon G\to\C$ be a finitely supported function. For $n\in G$, we define the  multiplicative difference operator by 
\[
\Delta_{ h} f(n) = \overline{f(n+h)} f(n).
\]
For a $k$-tuple $\vec h=(h_1,\dots,h_k)\in G^k$, we define
\begin{align}\label{eq:deltadef}
\Delta_{\vec h} f(n) =\Delta_{h_k}\cdots \Delta_{h_1} f(n) = \prod_{\omega\in \set{0,1}^k} \mathcal C^{|\omega|} f(n+\omega\cdot \vec h),
\end{align}
where $\mathcal C\colon z\to \bar z$ denotes complex conjugation and for a vector $\omega=(\omega_1,\dots,\omega_k)$ write $|\omega| =|\omega_1| +\cdots + |\omega_k|$.

If $f\colon G\to\C$ is a finitely supported function on an abelian group $G$, we define the unnormalised Gowers uniformity norm $\norm{f}_{\tilde{U}^k(G)}$ of order $k$ to be the quantity
\[
\norm{f}_{\tilde{U}^k(G)} =\biggbrac{\sum_{n\in G}\sum_{\vec h \in G^k} \Delta_{\vec h} f(n)}^{1/2^k}.
\]

We will need the Gowers--Cauchy--Schwarz inequality over $G$, which states that for any natural number  $k\geq 1$ and any family of functions $f_\omega\colon G\to \C$ for $\omega \in \set{0,1}^k$, we have
\begin{align}\label{eq:gcs2}
\Bigabs{\sum_{x,h_1,\dots,h_k\in G}\prod_{\omega\in \set{0,1}^k} \mathcal C^{|\omega|}f_\omega(x+\omega \cdot \vec h)} \leq \prod_{\omega \in \set{0,1}^k} \norm{f_\omega}_{\tilde{U}^k(G)}.
\end{align}
See~\cite[(4.2)]{TT}. The standard proof for finite abelian groups given in~\cite[Lemma B.2]{GT-linear} works also for infinite abelian groups.

Finally, for any function $f\colon\Z \to\C$ and any $N\geq 1$, we define the interval Gowers uniformity norm by
\[
\norm{f}_{U^k[N]} =\frac{\norm{f1_{[N]}}_{\tilde{U}^k(\Z)}}{\norm{1_{[N]}}_{\tilde{U}^k(\Z)}}. 
\]
We will also need a version of Gowers--Cauchy--Schwarz adapted to $[N]$. 

\begin{lemma}[Interval  Gowers--Cauchy--Schwarz]\label{le:gcs}
 Let $k$ be a natural number, and let $N\geq 1$. Let $f_\omega\colon \Z\to \C$ be functions supported on  $[N]$ for $\omega\in \{0,1\}^k$. Then
\begin{align}\label{eq:gcs}
\Bigabs{N^{-1} \sum_{x\in \mathbb{Z}}\E_{h_1,\dots,h_k \in [-N,N]} \prod_{\omega\in \set{0,1}^k} \mathcal C^{|\omega|}f_\omega(x+\omega \cdot \vec h)} \ll_k \prod_{\omega \in \set{0,1}^k} \norm{f_\omega}_{U^k[N]}.
\end{align}   
\end{lemma}

\begin{proof}
This follows from~\eqref{eq:gcs2} by taking $G=\mathbb{Z}$ and noting that $x+\omega \cdot \vec h\in [N]$ for all $\omega\in \{0,1\}^k$ implies $x,h_1,\ldots, h_k\in [-N,N]$.
\end{proof}

\subsection*{Acknowledgements}

JT was supported by funding from the European Union's Horizon Europe research and innovation programme under ERC grant agreement no. 101162746.
MW was supported by the Swiss National Science Foundation grant TMSGI2-2112.  

\section{Transferred inverse theorem for Gowers uniformity norms}\label{sec2}

In this section, we establish a transferred version of the quasipolynomial inverse theorem for the Gowers uniformity norm for unbounded functions with a pseudorandom majorant. This result matches the quasipolynomial bound of Leng--Sah--Sawhney~\cite{LSS} for $1$-bounded functions. For convenience, we briefly recall the relevant definitions relating to nilsequences here, while referring the reader to~\cite{LSS} for more detailed notation.

\begin{definition}[Filtered nilmanifold]\label{nilmanifold}

Let $s$ be a natural number and $d,M\geq 1$. We say that $G/\Gamma =(G/\Gamma, G_\bullet,\mathcal X)$ is a \emph{ nilmanifold  of degree at most $s$, dimension at most $d$  and complexity at most $M$} if the following conditions hold:
\begin{enumerate}
	\item $G$ is a connected, simply-connected nilpotent Lie group of dimension at most $d$;
	\item $\Gamma$ is a discrete, cocompact subgroup of $G$;
	\item $G_\bullet =(G_{(i)})_{i=0}^\infty$ is a filtration of $G$ of degree at most $s$;
	\item  $\mathcal X$ is  a Mal'cev basis for $G/\Gamma$ adapted to $G_\bullet$, and $\mathcal X$ is  $M$-rational.
\end{enumerate} 
\end{definition}

\begin{definition}[Nilsequences]\label{def-nil}
	Let $s,d,D$ be natural numbers and $M'\geq 1$.
	A \emph{polynomial orbit} of degree at most $s$, dimension at most $d$ and complexity at most $M'$ is any function $n \mapsto g(n)\Gamma$ from $\Z^D \to G/\Gamma$, where $(G/\Gamma, G_\bullet,\mathcal X)$ is a filtered nilmanifold of degree at most $s$, dimension at most $d$ and complexity at most $M'$ and $g\in \poly(\Z^D, G_\bullet)$ is a polynomial sequence, defined in~\cite[Definition 2.5]{LSS}.
	
	We define the \emph{Lipschitz norm} of a function $F\colon X\to \C$ on a metric space $(X,\rd_X)$ by
	\[
	\norm{F}_\lip = \norm{F}_{L^\infty(X)} + \sup_{\substack{x,y \in X\\ x\neq y}} \frac{|F(x)-F(y)|}{ \rd_X(x,y)}.
	\]
	
	 Let $K\geq 1$. A \emph{nilsequence of degree at most $ s$, dimension at most $d$ and complexity at most $(M', K)$} is any function $h\colon \Z^D \to\C$ of the form 
\[
h(n) = F(g(n)\Gamma),
\] 
where $n\mapsto g(n)\Gamma$ is a polynomial orbit of degree $\leq s$, dimension $\leq d$ and complexity $\leq M'$, and $F\colon G/\Gamma \to\C$ is a function with
\begin{align*}
\norm{F}_\lip\leq K   
\end{align*}
(where the Lipschitz norm is with respect to the metric on $G/\Gamma$ given by the Mal'cev basis $\mathcal{X}$, as in~\cite[Definition 3.4]{LSS}). 

Lastly, for $M\geq 1$, we say that a nilsequence $h$ has \emph{complexity at most $M$} if it has complexity at most $(M',K)$ for some $M',K\leq M$.
\end{definition}

With this notation, the quasipolynomial inverse theorem of Leng--Sah--Sawhney states the following.

\begin{theorem}[Quasipolynomial inverse theorem]\label{inverse}
Let $k \geq 2$ be a natural number, let $N\geq 2$, and let $0<\eta < 1/3$. Suppose that $f\colon [N]\to\C$ is a $1$-bounded function such that
\[
\norm{f}_{U^k[N]} \geq \eta.
\]
Then there exist a nilmanifold $G/\Gamma$ of degree at most $k-1$, dimension at most $d$, and complexity at most $M$; a Lipschitz function $F\colon G/\Gamma \to \C$ with $\norm{F}_\lip=1$; and a polynomial map $g\colon \Z \to G$ such that
\[
\Bigl|\E_{n\in [N]} f(n)\,\overline{F(g(n)\Gamma)}\Bigr| \geq \eps,
\]
where
\[
d \leq (\log(1/\eta))^{O_k(1)} \quad \text{and} \quad \eps^{-1},\, M \leq \exp\!\bigl((\log(1/\eta))^{O_k(1)}\bigr).
\]
\end{theorem}

\begin{proof}
This is \cite[Theorem~1.2]{LSS}, with the normalisation adjusted to ensure that the Lipschitz norm is $1$. This is achieved by replacing $\eps$ with $\eps/K$ and then setting $K=1$ in their statement.	
\end{proof}

By combining the densification arguments of~\cite{CFZ} and~\cite{TT} with Theorem~\ref{inverse}, the  main result of this section establishes an inverse theorem for unbounded functions with a pseudorandom majorant. This result refines \cite[Theorem 8.3]{TT} by improving the parameter dependence from exponential to quasipolynomial. We first introduce a lemma that computes the complexity of a product of nilsequences.

\begin{lemma}[Product of nilsequences]\label{pro-nil}

Let $m,s\in \N$ and let $d,M\geq 1$. Suppose that $\psi_1,\dots,\psi_m\colon \Z \to\C$ are nilsequences of degree  $s$, dimension at most $d$, and complexity at most $M$. Then $\prod_{1\leq j\leq m}\psi_j$ is a nilsequence of degree at most $s$, dimension at most $md$, and complexity at most $(m+1)M^m$.
\end{lemma}

\begin{proof}

Write $\psi_j(n) =F_j(g_j(n)\Gamma_j)$; by Definition~\ref{def-nil}, we may assume that $F_j\colon G_j/\Gamma_j\to\C$ has Lipschitz norm at most $M$; $(G_j/\Gamma_j, G_{j\bullet},\mathcal X_j)$ is a filtered nilmanifold of degree  $s$, dimension at most $d$ and $\mathcal X_j$ is $M$-rational; and $g_j\in\poly(\Z, G_{j\bullet})$ is a polynomial sequence.

Set $G\coloneqq G_1\times \cdots \times G_m$ and $\Gamma\coloneqq  \Gamma_1\times \cdots \times \Gamma_m$ equipped with the product filtration $G_\bullet = G_{1\bullet}\times \cdots G_{m\bullet}$ (taken componentwise), and equip $G/\Gamma$ with the product metric. Define
\[
g(n)\coloneqq (g_1(n),\dots,g_m(n))\in G
\quad\text{and}\quad
F(x_1,\dots,x_m)\coloneqq \prod_{j=1}^m F_j(x_j).
\]
Then $g\in \poly(\Z,G_\bullet)$ and $\prod_{1\leq j\leq m}\psi_j = F(g(\cdot)\Gamma)$.

As an analogue of \cite[Fact~3.9]{LSS}, define
\[
\mathcal X \coloneqq \bigcup_{1\le j\le m}
\bigset{
\underbrace{(0,\dots,0, X_j, 0,\dots,0)}_{\text{only the $j$-th entry is nonzero}}
\colon X_j\in \mathcal X_j}.
\]
Then $\mathcal X$ forms a basis of $\log G$. It follows from~\cite[Definition~3.5]{LSS} and Definition~\ref{nilmanifold} that $(G/\Gamma,G_\bullet)$ has degree at most $s$, and its dimension satisfies
\[
\dim(G/\Gamma)\le \sum_{1\le j\le m}\dim(G_j/\Gamma_j)\le md.
\]

Moreover, by~\cite[Definitions~3.3 and~3.5]{LSS}, for any $X,X'\in\mathcal X$ the Lie bracket behaves as follows. 
If there exists some $1\le j\le m$ and $X_j,X'_j\in\mathcal X_j$ such that
\[
X=(0,\dots,0,X_j,0,\dots,0),\qquad
X'=(0,\dots,0,X'_j,0,\dots,0),
\]
then
\[
[X,X']=(0,\dots,0,[X_j,X'_j],0,\dots,0),
\]
so the rational structure constants $c_{ijk}$ agree with those of $G_j$ in this case. 
Otherwise, if $X$ and $X'$ are supported on distinct coordinates (i.e., there exist $i\neq j$ with
$X=(0,\dots,0,X_i,0,\dots,0)$ and $X'=(0,\dots,0,X'_j,0,\dots,0)$), then $[X,X']=0$, so all cross terms vanish. 
Therefore, $G/\Gamma$ has complexity at most $M$ with respect to the Mal'cev basis $\mathcal X$.

Finally, one may verify from Definition~\ref{def-nil} that
\begin{align*}
\|F\|_{\lip}
&\le \prod_{1\le j\le m}\|F_j\|_{\infty}
  + \sup_{\substack{\vec x,\vec y \in G/\Gamma\\ \vec x\neq \vec y}} \frac{\bigl| \prod_{1\le j\le m}F_j(x_j\Gamma_j)
  - \prod_{1\le j\le m}F_j(y_j\Gamma_j)\bigr|}{d_{\mathcal X}((x_1,\dots,x_m),(y_1,\dots,y_m))}\\
  &\leq  M^m +m \sup_{1\leq j\leq m} \prod_{i\neq j}\|F_i\|_{\infty} \sup_{\substack{\vec x,\vec y \in G/\Gamma\\ \vec x\neq \vec y}}\frac{|F_j(x_j\Gamma_j)-F_j(y_j\Gamma_j)|}{d_{\mathcal X}((x_1,\dots,x_m),(y_1,\dots,y_m))}.
\end{align*}
Moreover, by the definition of the Lipschitz norm, for each $1\leq j\leq m$ and each $x_j\neq y_j\in G_j/\Gamma_j$ one has
\[
|F_j(x_j\Gamma_j)-F_j(y_j\Gamma_j)|\leq d_{\mathcal X_j}(x_j,y_j)\cdot\norm{F_j}_\lip. 
\]
Combining these estimates, and using the assumption that $\norm{F_i}_\lip\leq M$ for all $1\leq i\leq m$ we obtain
\[
\|F\|_{\lip}\leq M^m+ mM^m \sup_{1\leq j\leq m}\sup_{\substack{\vec x,\vec y \in G/\Gamma\\ \vec x\neq \vec y}}\frac{d_{\mathcal X_j}(x_j,y_j)}{d_{\mathcal X}((x_1,\dots,x_m),(y_1,\dots,y_m))}\leq (m+1)M^m,
\]
which yields the claimed complexity bound.

\end{proof}

\begin{theorem}[Transferred inverse theorem---quasipolynomial bounds]\label{tra-inv}

Let $k\geq 2$ be a natural number, let $N\geq 2$, and let $0<\eps  <1/3$ be a parameter. Let $\nu\colon [N] \to \R_{\geq 0}$ be a function that satisfies
\begin{equation}\label{eq:nuassumption}
\norm{\nu -1}_{U^{2k}[N]} \leq \exp \bigbrac{ -(\log(1/\eps))^{D_k}}    
\end{equation}
for some constant  $D_k>1$   sufficiently large in terms of $k$.
Suppose that $f\colon [N] \to \C$ satisfies $|f|\leq \nu$ and 
\[
\norm{f}_{U^k[N]} \geq \eps.
\]
Then there exist a constant $0<C_k<D_k/10$ depending only on $k$; a  filtered nilmanifold $G/\Gamma$ of degree $k-1$, dimension at most $(\log(1/\eps))^{C_k}$, and complexity at most $\exp\bigbrac{(\log(1/\eps ))^{C_k} }$; a Lipschitz function $F\colon  G/\Gamma \to\C$ with $\norm{F}_\lip =1$; and a polynomial map $g\colon \Z \to G$ such that
\[
\bigabs{\E_{n\in [N]} f(n) \overline{F(g(n)\Gamma)}} \geq  \exp\bigbrac{-(\log(1/\eps))^{C_k} }.
\]	
\end{theorem}

\begin{proof}

The strategy is to adapt the densification method of~\cite{CFZ}, together with the inverse theorem for the Gowers $U^k$ norm for $1$-bounded functions (Theorem~\ref{inverse}) and an induction argument. The overall structure closely follows the proof of \cite[Theorem 8.3]{TT}.

We  extend $f$ and $\nu$ to functions on $\mathbb{Z}$ supported on $[N]$. We shall establish by induction on $\mathcal{I}'\subseteq \{0,1\}^k$, with $\vec 1\in \mathcal{I}'$, the following statement.

\textbf{Hypothesis}  $\textnormal{H}(\mathcal{I}')$. Set $\mathcal I''$ as the complement of $\mathcal I'$ in $\{0,1\}^k$.
 We claim that there exists a constant $C_{\mathcal I',k}>0$, depending only on $k$ and $\mathcal I'$, such that for every $\omega_2\in \mathcal I''$ there exists a nilsequence $\psi_{\omega_2} = F_{\omega_2}(g_{\omega_2}(\cdot)\Gamma_{\omega_2})$, where $g_{\omega_2}\colon\Z \to G_{\omega_2}$ is a polynomial map,
$F_{\omega_2}$ is a $1$-Lipschitz  function that is  defined on a nilmanifold $G_{\omega_2}/\Gamma_{\omega_2}$ of degree $k-1$, complexity $\leq \exp ((\log(1/\eps))^{C_{\mathcal I',k}})$, and dimension $\leq (\log(1/\eps)   )^{C_{\mathcal I',k}}$,  and such that we have 
\begin{multline}\label{bias-ass}
	\biggabs{N^{-1}\sum_{n\in \Z}\E_{\vec h\in[-N,N]^{k}} \prod_{\omega_1\in \mathcal I'} \mathcal C^{|\omega_1|} f(n+\omega_1\cdot \vec h) \prod_{\omega_2\in \mathcal I''} \psi_{\omega_2} (n+\omega_2 \cdot \vec h)} \\\gg_{\mathcal{I}'} \exp\!\bigl(-(\log(1/\eps))^{C_{\mathcal I',k}}\bigr).
\end{multline}

Note that $\textnormal{H}(\{0,1\}^k)$ holds by the assumption of the theorem. Indeed, by assumption, the left-hand side equals
\[
\Bigabs{N^{-1}\sum_{n\in \Z}\E_{\vec h\in[-N,N]^{k}} \prod_{\omega\in \set{0,1}^k} \mathcal C^{|\omega|} f(n+\omega\cdot \vec h)}\gg_k\norm{f}_{U^k[N]}^{2^k}\gg_k\eps^{2^k}.
\]

Note also that if $\textnormal{H}(\{\vec 1\})$ holds, then assuming that $C_k>1$ is large enough in terms of $k$, we obtain
\[
\Bigabs{N^{-1}\sum_{n\in \Z}\E_{\vec h\in[-N,N]^{k}}\mathcal C^{|\vec 1|} f (n+\vec 1\cdot\vec h) \prod_{\omega \in \set{0,1}^k\setminus\{\vec 1\}} \psi_\omega (n+\omega \cdot \vec h)} \gg_k \exp\bigbrac{-(\log(1/\eps))^{C_k/2} }.
\]
After the change of variables $n+\vec 1\cdot \vec h \to n$ and an application of the pigeonhole principle, we can find some $\vec h_0 \in [-N,N]^{k}$ such that
\begin{equation}\label{eq:H(1)}
\Bigabs{N^{-1}\sum_{n\in \Z} \mathcal C^{|\vec 1|} f(n) \prod_{\omega \in \set{0,1}^k\backslash \set{\vec 1}} \psi_\omega (n+(\omega-\vec 1) \cdot \vec h_0)} \gg_k \exp\bigbrac{-(\log(1/\eps) )^{C_k/2} }.
\end{equation}
Since $C_k$ is large enough, we can absorb the constant in the $\gg_k$ notation by adjusting the exponent and say that~\eqref{eq:H(1)} is $\geq \exp\bigbrac{-(1/2)(\log(1/\eps))^{C_k} }$. Moreover, since  $\supp(f)\subset [N]$, we can further adjust the exponent and reduce the inequality to
\[
\Bigabs{\E_{n\in [N]} f(n) \prod_{\omega \in \set{0,1}^k\backslash \set{\vec 1}}\mathcal C^{|\vec 1|} \psi_\omega (n+(\omega-\vec 1) \cdot \vec h_0)} \geq \exp\bigbrac{-(\log(1/\eps) )^{C_k} }.
\]

The theorem now follows by taking the tensor product of the nilsequences obtained above and then rescaling the resulting Lipschitz function to have norm 1. Moreover, by Lemma~\ref{pro-nil}, the dimension and complexity of the resulting product nilsequence increase at most polynomially, as claimed.

\noindent\textbf{Proof of the induction step.}
Now, by induction, the remaining task is to show that $\textnormal{H}(\mathcal{I}')$ implies $\textnormal{H}(\mathcal{I}'\setminus \omega')$ for any $\omega'\in \mathcal{I}'\setminus \{\vec 1\}$.

Let us define the dual function $H_{\omega'}$ by
\[
H_{\omega'}(n) =\E_{\vec h\in [-N,N]^{k}} \prod_{\omega_1\in \mathcal I'\backslash\set{\omega'} } \mathcal C^{|\omega_1|}f(n+(\omega_1-\omega')\cdot \vec h) \prod_{\omega_2\in \mathcal I''} \psi_{\omega_2} (n+(\omega_2-\omega') \cdot \vec h).
\]
Let $\eps_{\mathcal{I}'}\gg_k \exp\bigbrac{-(\log(1/\eps))^{C_{\mathcal I',k}}}$ be such that the left-hand side of~\eqref{bias-ass}  is $\geq \eps_{\mathcal{I}'}$.

From assumption~\eqref{bias-ass}, we have
\[
\eps_{\mathcal{I}'} \leq \left|N^{-1}\sum_{n\in \Z} f(n) H_{\omega'} (n)\right|.
\]
Observe that we have the bound 
\begin{align}\label{eq:monotone}
\mathbb{E}_{n\in [N]}\nu(n) \ll_k \|\nu\|_{U^{2k}[N]}\ll_k 1,    
\end{align}
which follows from Lemma~\ref{le:gcs} applied to one copy of $\nu$ and $2k-1$ copies of $1_{[(2k+1)N]}$.
Now, applying the Cauchy--Schwarz inequality and using the assumption $|f| \leq \nu$, we obtain
\[
\eps_{\mathcal{I}'} ^2 \ll_k N^{-1}\sum_{n\in \Z} \nu(n) |H_{\omega'} (n)|^2 = N^{-1}\sum_{n\in \Z} |H_{\omega'} (n)|^2 + N^{-1}\sum_{n\in \Z} (\nu -1) (n) |H_{\omega'} (n)|^2.
\]
Since $\norm{\nu-1}_{U^{2k}[N]} \ll \exp (-(\log(1/\eps))^{D_k})$ for some sufficiently large $D_k$ depending on $k$, and since $1/\eps_{\mathcal{I}'}\ll_k \exp((\log(1/\eps))^{O_k(1)})$,  the Gowers--Cauchy--Schwarz inequality gives us
\begin{align*}
N^{-1}\sum_{n\in \Z} (\nu -1) (n) |H_{\omega'} (n)|^2 &\ll_k \norm{\nu -1}_{U^{2k}[N]} \norm{\nu +1}_{U^{2k}[N] }^{2^{2k}-1}\\
&\ll_k \exp (-(\log(1/\eps))^{D_k})(\|\nu-1\|_{U^{2k}[N]}+2)^{2^{2k}-1} \\
&\ll_k \exp (-(\log(1/\eps))^{D_k}).
\end{align*}
Combining these estimates, and taking $D_k$ large enough, we deduce
\begin{equation}\label{eq:F2bound}
N^{-1}\sum_{n\in \Z}|H_{\omega'}(n)|^2\gg_k\varepsilon_{\mathcal{I}'}^2.
\end{equation}
Now, introduce the $1$-bounded truncated function
\[
\widetilde{H}_{\omega'}(n)=\min\{|H_{\omega'}(n)|,1\}\cdot \sgn(H_{\omega'}(n)),    
\]
where $\sgn(z)=z/|z|$ for $z\in \mathbb{C}\setminus\{0\}$ and $\sgn(0)=0$. 
By the Cauchy--Schwarz inequality, we have
\begin{align*}
N^{-1}\sum_{n\in \Z}\overline{H_{\omega'}(n)}\widetilde{H}_{\omega'}(n)& = N^{-1}\sum_{n\in \Z}|H_{\omega'}(n)|^2-N^{-1}\sum_{n\in \Z}\overline{H_{\omega'}(n)}(H_{\omega'}(n)-\widetilde{H}_{\omega'}(n))\\
&\geq N^{-1}\sum_{n\in \Z}|H_{\omega'}(n)|^2-N^{-1}\sum_{n\in \Z}(|H_{\omega'}(n)-\widetilde{H}_{\omega'}(n)|^2\\
&\quad +|H_{\omega'}(n)-\widetilde{H}_{\omega'}(n)|)\\
&\geq N^{-1}\sum_{n\in \Z} |H_{\omega'}(n)|^2-N^{-1}\sum_{n\in \Z}|H_{\omega'}(n)-\widetilde{H}_{\omega'}(n)|^2\\
&\quad -\left(N^{-1}\sum_{n\in \Z}|H_{\omega'}(n)-\widetilde{H}_{\omega'}(n)|^2\right)^{1/2}.
\end{align*}
On the other hand,
\begin{align*}
|H_{\omega'}(n)-\widetilde{H}_{\omega'}(n)|&\leq \max\{|H_{\omega'}(n)|-1,0\}\\
&\leq \Bigabs{\E_{\vec h\in [-N,N]^{k}} \prod_{\omega_1\in \mathcal I'\backslash\set{\omega'} } \nu(n+(\omega_1-\omega')\cdot \vec h)-1}\\
&\leq \sum_{\substack{\mathcal{J}\subseteq \mathcal{I}'\setminus \{\omega'\}\\ \mathcal{J}\neq \emptyset}} \biggabs{\E_{\vec h\in [-N,N]^{k}}\prod_{\omega_1\in \mathcal{J}}(\nu-1)(n+(\omega_1-\omega')\cdot \vec h)}.
\end{align*}
Building on this, we can conclude from the Gowers--Cauchy--Schwarz inequality and the assumption~\eqref{eq:nuassumption}  that
\begin{equation}\label{eq:F-tildeF}
N^{-1}\sum_{n\in \Z}|H_{\omega'}(n)-\widetilde{H}_{\omega'}(n)|^2\ll_k \|\nu-1\|_{U^{2k}[N]}\ll_k \exp(-(\log(1/\varepsilon))^{D_k}).
\end{equation}
Now, in view of~\eqref{eq:F2bound} and the fact that $D_k$ is large enough in~\eqref{eq:nuassumption}, we have
\begin{align*}
\varepsilon_{\mathcal{I}'}^2 &\ll_k N^{-1}\sum_{n\in \Z}\overline{H_{\omega'}(n)}\widetilde{H}_{\omega'}(n)\\
&\ll_k N^{-1}\sum_{n\in \Z} \widetilde{H}_{\omega'}(n) \E_{\vec h\in [-N,N]^{k}} \prod_{\omega_1\in \mathcal I'\backslash\set{\omega'} } \mathcal C^{|\omega_1|+1}f(n+(\omega_1-\omega')\cdot \vec h)\\
&\qquad \times\prod_{\omega_2\in \mathcal I''}\overline{ \psi_{\omega_2} (n+(\omega_2-\omega') \cdot \vec h)}.
\end{align*}
We then conclude from Gowers--Cauchy--Schwarz inequality  that
\[
\eps_{\mathcal{I'}}^{O_k(1)} \ll_k \|\widetilde{H}_{\omega'}\|_{U^{k}[N]}.
\]
Applying Theorem~\ref{inverse} to the $1$-bounded function $\widetilde{H}_{\omega'}$, we obtain a nilsequence  $\psi_{\omega'} = F_{\omega'} (g_{\omega'} (\cdot) \Gamma_{\omega'})$, where $F_{\omega'}\colon G_{\omega'}/\Gamma_{\omega'}\to \C$ is $1$-Lipschitz and the nilmanifold $G_{\omega'}/\Gamma_{\omega'}$ has the required degree, complexity, and dimension bounds,  and such that
\[
\Bigabs{N^{-1}\sum_{n\in \Z}  \widetilde{H}_{\omega'} (n) \overline{\psi_{\omega'} (n)}}  \gg_k \exp \bigbrac{-(\log(1/\eps_{\mathcal{I}'}))^{O_k(1)}} \gg_k \exp\bigbrac{-(\log(1/\eps ))^{O_k(C_{\mathcal I',k})}}.
\]
Choose a number $C_{\mathcal I'\backslash\set{\omega'},k}\gg_k C_{\mathcal I',k}$  so that the above left-hand side is in fact bounded below by $\exp (-(\log(1/\eps))^{C_{\mathcal I'\backslash\set{\omega'},k}})$.
By the $1$-boundedness of $\psi_{\omega'}$, the Cauchy--Schwarz inequality, the estimate~\eqref{eq:F-tildeF}, and the assumption that $D_k$ is large enough, we conclude that also
\[
\Bigabs{N^{-1}\sum_{n\in \Z}  H_{\omega'} (n) \overline{\psi_{\omega'} (n)}} \gg_k \exp (-(\log(1/\eps))^{C_{\mathcal I'\backslash\set{\omega'},k}}),
\]
so 
\begin{align}\begin{split}
&\Bigabs{N^{-1}\sum_{n\in \Z} \overline{\psi_{\omega'} (n)} \E_{\vec h\in [-N,N]^{k}} \prod_{\omega_1\in \mathcal I'\backslash\set{\omega'} } \mathcal C^{|\omega_1|}f(n+(\omega_1-\omega')\cdot \vec h) \prod_{\omega_2\in \mathcal I''} \psi_{\omega_2} (n+(\omega_2-\omega') \cdot \vec h) }\\
&\gg_k \exp (-(\log(1/\eps))^{C_{\mathcal I'\backslash\set{\omega'},k}}).
\end{split}
\end{align}
This completes the proof of the induction claim after the linear change of variables $n\to n+\omega'\cdot \vec h$ and possibly renaming $\overline{\psi_{\omega'}}$ as $\psi_{\omega'}$.
Now the theorem is proved.
\end{proof}

\section{Transferred regularity lemmas (dense model)}\label{sec3}

In this section, we establish a quantitative transferred regularity lemma. Roughly speaking, it asserts that if a function $f$ is bounded by a pseudorandom majorant, then $f$ can be decomposed into a structural component and a uniform component. The corresponding $1$-bounded case was studied in~\cite{GT-regularity}. We emphasise that our notion of a structural component differs from that in~\cite{GT-regularity}, where the structural part is stronger in the sense that it is a nilsequence. When the structural component is $1$-bounded, this corresponds to the so-called \emph{dense model}, which has been investigated qualitatively in~\cite{GT-prime} and quantitatively (by different methods) in~\cite{Go-dense} and~\cite{reing}. We begin by recalling some preliminary definitions.

\begin{definition}[Factor]
Let $N\geq 1$. A \emph{factor} $\CB$ of the set $[N]$ is a partition of $[N]$ into disjoint nonempty subsets. That is, $[N] = \bigsqcup_{B \in \CB} B$. For $x \in [N]$, we denote by $\CB(x) \in \CB$ the unique \emph{atom} (or \emph{cell}) containing $x$. Given two factors $\CB$ and $\CB'$, we say that $\CB'$ \emph{refines} $\CB$ if every atom of $\CB$ is a union of atoms of $\CB'$. The \emph{join} of factors $\CB_1,\dots,\CB_d$ is the partition
\[
\CB_1\vee \cdots \vee \CB_d = \set{B_1\cap \cdots \cap B_d\colon  B_i\in \CB_i},
\]
with empty atoms discarded.
\end{definition}

\begin{definition}[Factor of complexity $d$ and resolution $K$]\label{factor-res}
Let $N\geq 1$ and let $d,s,M$ be natural numbers.
Let $h\colon [N]\to \R$ be a real-valued function and let $K \geq 1$ be a parameter. The \emph{factor induced by $h$ of resolution $K$} is defined by
\[
\CB_{h,K} \coloneqq  \left\{ \Bigl\{ n \in [N] \colon  \tfrac{j}{K} < h(n) \leq \tfrac{j+1}{K} \Bigr\}\colon j\in \mathbb{Z}\right\},
\]
with empty atoms discarded.

A \emph{factor of degree $s$, dimension at most $d$ and complexity at most $(M,1)$ and resolution $K$} is a join of the form
\[
\CB = \CB_{h_1,K} \vee \cdots \vee \CB_{h_{m},K},
\]
for some integer $m \leq M$ and some nilsequences $h_1,\ldots, h_{m}\colon [N]\to \mathbb{R}$ of degree $s$, dimension at most $d$ and complexity at most $(M,1)$. In particular, each $h_i$ is 1-bounded. We refer to the factors $\CB_{h_i,K}$ as the \emph{generators} of $\CB$.
\end{definition}

Before stating the main result of this section, we address a subtle issue with the above decomposition: many points may lie near the boundaries between atoms, which is undesirable. To avoid this, we adopt the notion of \emph{regular factors}, following \cite[Definition~3.5]{LSS}.

\begin{definition}[Regular factor]\label{reg-fac}
 Let $N,K\geq 1$, $C>0$, and let $h\colon [N]\to \mathbb{R}$ be a function. The factor $\CB_{h,K}$ is \emph{$C$-regular} if
\[
\sup_{r>0} \frac{1}{2r} \frac{\#\set{ n\in [N]\colon  \norm{K \cdot h(n)}_\T \leq r}}{N}\leq C.
\]
More generally, a factor $\CB$ of dimension at most $d$, complexity at most $(M,1)$ and resolution $K$ is said to be $C$-\emph{regular} if all of its generators are $C$-regular.	
\end{definition}

The following fact guarantees that any induced factor can be made regular after a small shift.

\begin{lemma}\label{lem-reg}
Let $N,K\geq 1$. There exists an absolute constant $C>0$ such that the following statement holds. Given any function $h\colon [N]\to \R$, there exists a shift $t\in [0,\frac{1}{K})$ such that $\CB_{h-t, K}$ is $C$-regular.	
\end{lemma}

\begin{proof}
We define a function $M_{h,K}\colon [0,1] \to \R_{\geq 0}\cup\{\infty\}$ by setting
\[
M_{h,K} (t) = \sup_{r>0} \frac{1}{2r} \frac{\# \set{ n\in [N]\colon \norm{K\cdot h(n) -t}_\T \leq r}}{N}.
\]
It follows from the Hardy--Littlewood maximal inequality (see~\cite[Proof of Corollary 2.3]{GT-regularity}) that for any $\lambda>0$
\[
\meas \set{t\in [0,1)\colon M_{h,K} (t)\geq \lambda} \ll \lambda^{-1}.
\]

Thus, taking $\lambda=C$, it follows that for all real numbers $t\in [0,1)$ outside of a set of Lebesgue measure $O(C^{-1})$ we have the inequality
\[
 \sup_{r>0} \frac{1}{2r} \frac{1}{N} \# \set{ n\in [N]\colon \norm{K\cdot h(n) -t}_\T \leq r} \leq C.
\]
In light of Definition~\ref{reg-fac}, this implies that, when $C$ is large enough, there exists some $t\in [0,1)$ such that $\CB_{h-\frac{t}{K},K}$ is $C$-regular. The lemma follows by renaming $t/K$ as $t\in [0,1/K)$.
\end{proof}


\subsection{Transferred regularity lemma I}

In this subsection, we prove Lemma~\ref{tra-wea}, which is a transferred regularity lemma. In the following subsection, we simplify this result to a more applicable form, which is Proposition~\ref{tra-wea-2}.

\begin{definition}[Conditional expectation]\label{cond-ex}
Let $N\geq 1$. Let $\CB$ be a factor of $[N]$ and let $f \colon  [N] \to \C$ be a function. The \emph{conditional expectation} of $f$ with respect to $\CB$ is the function $\Pi_\CB f \colon  [N] \to \C$ defined by
\[
\Pi_\CB f(x) = \E_{y\in \CB(x)} f(y).
\]	
\end{definition}

By construction, $\Pi_\CB f$ takes constant values on the atoms of $\CB$; that is, $\Pi_\CB f(x) = \Pi_\CB f(y)$ whenever $\CB(x) = \CB(y)$. We will frequently use the identity
\begin{align}\label{eq:conditionalexp}
 \mathbb{E}_{x\in B}\Pi_{\mathcal{B}}f(x)=\mathbb{E}_{x\in B}f(x)   
\end{align}
for $B\in \mathcal{B}$.

Our transferred weak regularity lemma (Lemma~\ref{tra-wea}) relies on the following orthogonality identity.

\begin{lemma}[Pythagoras's theorem]\label{pythagoras}
Let $N\geq 1$. Let $\CB$ and $\CB'$ be factors of $[N]$, with $\CB'$ refining $\CB$. Let $f\colon [N]\to\C$ be a function. Then
\[
\norm{\Pi_{\CB'} f}_2^2 = \norm{\Pi_{\CB} f}_2^2 + \norm{\Pi_{\CB'} f - \Pi_\CB f}_2^2.
\]	
\end{lemma}

\begin{proof}
We may write $\Pi_{\CB'} f = \Pi_\CB f + (\Pi_{\CB'} f - \Pi_\CB f)$. By employing the identity $|a+b|^2 = |a|^2 + 2\textnormal{Re}(a\bar{b})+ |b|^2$, it suffices to show that 
\begin{align}\label{orthog}
\ang{ \Pi_{\CB} f , \Pi_{\CB'} f - \Pi_{\CB} f}=0.	
\end{align}
Let $[N] = \bigsqcup_{B \in \CB} B$, and choose a representative $x_B \in B$ for each atom $B \in \CB$. Since $\Pi_{\CB} f$ is constant on each atom $B$, we have $\Pi_{\CB} f(x) = \Pi_{\CB} f(x_B)$ for all $x \in B$. Hence
\[
\E_{x\in [N]} \Pi_\CB f(x) \overline{\Pi_{\CB'} f(x)} = \sum_{B \in \CB} \frac{|B|}{\lfloor N\rfloor} \Pi_{\CB} f(x_B) \E_{x\in B} \overline{\Pi_{\CB'} f(x)}.
\]
But if $x \in B$, then $\CB'(x) \subseteq \CB(x)=B$. Averaging over $x\in B$ gives
\[
\E_{x\in B} \Pi_{\CB'} f(x) = \E_{x\in B} \E_{y\in \CB'(x)} f(y) =\E_{y\in B} f(y)  = \E_{x\in B}\Pi_\CB f(x).
\]
Hence, combining the above two estimates yields
\begin{align*}
\E_{x\in [N]} \Pi_\CB f(x) \overline{\Pi_{\CB'} f(x)} 
&= \sum_{B\in \CB} \frac{|B|}{\lfloor N\rfloor} \Pi_\CB f(x_B) \E_{x\in B}\overline{\Pi_\CB f(x)}\\
&= \frac{1}{N}  \sum_{B\in \CB} \sum_{x\in B}\abs{\Pi_\CB f(x)}^2\\
&=\norm{\Pi_{\CB}f}_2^2,
\end{align*}
which establishes the orthogonality~\eqref{orthog}. This completes the proof.
\end{proof}

\begin{lemma}[Transferred weak regularity lemma]\label{tra-wea}
Let $N\geq 1$, let $C\geq 1$ be a large constant, let $k\geq 2$ be a natural number, and let $0<\eps<1/3$ be a parameter. Suppose $\nu\colon [N]\to\R_{\geq 0}$ satisfies
\begin{align}\label{pseudo-1}
\norm{\nu -1}_{U^{2k}[N]} \leq \exp \bigbrac{ -(\log(1/\eps))^{D_k}}
\end{align}
for some constant $D_k>1$ sufficiently large in terms of $k$. 
Let $f\colon [N]\to \C$ be a function with $|f|\leq \nu$.

Assume that for every $C$-regular factor $\CB$ of degree $k-1$, dimension at most $(\log(1/\eps))^{D_k/10}$, complexity at most $(\exp ((\log(1/\eps))^{D_k/10}), 1)$, and resolution at most $\exp ((\log(1/\eps))^{D_k/10})$, the following hold:
\begin{enumerate}
	\item ($\ell^2$-norm boundedness).
\begin{equation}\label{l2-bd}
\|\Pi_{\CB} f  \|^2_2 \leq 10;
\end{equation}
\item there exists a set $\Omega_\CB \subset[N]$, which is a union of atoms of $\CB$, such that
\begin{align}\label{excep-ass-3.7}
	\Bigabs{\E_{n\in [N]} (|f|+1)(n) 1_{[N]\backslash \Omega_\CB}(n)} \leq \exp(-(\log(1/\eps))^{D_k})
\end{align}
and such that
\begin{equation}\label{1-bdd-assump}
\abs{\Pi_\CB f(n)} \leq 2 \text{ whenever } n\in \Omega_\CB.	
\end{equation}
\end{enumerate}

Then there exists a number $C_k\in (1,D_k/10)$, and parameters
\[
d\leq (\log(1/\eps))^{C_k}, \qquad M,T,K\leq \exp((\log(1/\eps))^{C_k}),
\]
together with a factor $\CB$ which is a $T$-fold join of $C$-regular factors $\CB_{h_i,K}$, where each has resolution $K$ and each $h_i$ is a nilsequence of degree $k-1$, dimension at most $d$, and complexity at most $(M,1)$, and a set $\Omega\subset[N]$ that is a union of atoms of $\CB$ such that
\[
\Bigabs{\E_{n\in [N]} (|f|+1)(n) 1_{[N]\backslash \Omega}(n)} \leq \exp(-(\log(1/\eps))^{D_k})
\]
and
\[
\norm{(f -\Pi_\CB f)\cdot1_\Omega }_{U^{k}[N]} \leq \eps.
\]	
\end{lemma} 

\begin{proof}
The overall strategy is to apply the \emph{energy increment argument}. 

Suppose that $C_k$ is sufficiently large in terms of $k$ and that $D_k$ is sufficiently large in terms of $C_k$. The following is the key claim powering this increment argument.

\textbf{Key claim.} Let $f\colon [N]\to \mathbb{C}$ be as in the lemma and set 
$$K=\exp((\log(1/\eps))^{C_k})/10,$$
let $d\in \mathbb{N}$, and let $C>0$ be a large absolute constant.
Let $\mathcal{B}$ be a $C$-regular factor of $[N]$ of complexity $(d+\frac{1}{2}\exp((\log(1/\varepsilon))^{C_k}),1)$ and resolution $K$. Then at least one of the following alternatives holds.
\begin{enumerate}
\item[(i)] 
 There exists a set $\Omega_\CB\subset[N]$ with properties~\eqref{excep-ass-3.7}--\eqref{1-bdd-assump}, and such that  
\begin{align} \label{assump-tr}
\norm{ (f-\Pi_{\CB} f)\cdot 1_{\Omega_\CB} }_{U^k[N]} \leq \eps.	
\end{align}
\item[(ii)]
There exists a refinement $\mathcal{B}'$ of $\mathcal{B}$ having  complexity at most $(d+1+\frac{1}{2}\exp((\log(1/\varepsilon))^{C_k}),1)$ and resolution $K$, and with the generators of $\mathcal{B}'$ being $C$-regular factors, such that
\[
\norm{ \Pi_{\CB'}f - \Pi_{\CB}f }_2^2  \geq \exp\bigbrac{- (\log(1/\eps) )^{C_{k}}}/101
\]
and
\[
\mathcal{B}'=\mathcal{B}\vee \mathcal{B}_{h,K}
\]
for some nilsequence $h$ of degree $k-1$, dimension $\leq (\log(1/\eps))^{C_k}$ and complexity at most $(\exp((\log(1/\eps))^{C_k})/2,1)$.
\end{enumerate}

\textbf{Proof of key claim.}

Without loss of generality, we may assume that the inequality~\eqref{assump-tr} fails. 

By assumption, there exists a set $\Omega_{\CB}\subseteq [N]$ such that $|\Pi_{\CB}f(n)|\le 2$ for all $n\in \Omega_{\CB}$. 
Thus, $\nu+2$ is a majorant of $(f-\Pi_{\CB} f)\cdot 1_{\Omega_\CB}$ on $[N]$. We normalise this majorant by setting $\tilde \nu =\frac{\nu+2}{3}$, so $\tilde \nu$ satisfies the pseudorandomness assumption~\eqref{pseudo-1}. For convenience, we abuse the notation by using $\nu$ to denote $\tilde\nu$ in the remainder of the proof. 

We first note that, if~\eqref{assump-tr} fails, then one can apply Theorem~\ref{tra-inv} to the function $(f-\Pi_{\CB} f)\cdot 1_{\Omega_\CB}$ (viewed as a function on $[N]$ by extending it by 0 outside $\Omega_\CB$). This yields a nilsequence $F_{\mathcal{B}}(g_{\mathcal{B}}(n)\Gamma_{\mathcal{B}})$ of degree $k-1$, of dimension at most $(\log(1/\eps))^{C_{k}}$ and of complexity at most $(\exp((\log(1/\varepsilon))^{C_{k}/2})/2,1)$, such that
\begin{align}\label{eq:fPiB}
 \bigabs{ \E_{n\in [N]} \bigbrac{ f-\Pi_{\CB} f }(n) 1_{\Omega_\CB}(n) \overline{F_{\mathcal{B}}(g_{\mathcal{B}}(n)\Gamma_{\mathcal{B}})}} \geq \exp\bigbrac{- (\log(1/\eps))^{C_{k}/2}}/2. 
\end{align}
Taking the real or imaginary part if necessary, we may assume $F_{\mathcal{B}}$ is real-valued, and then the bound on the right-hand side of~\eqref{eq:fPiB} becomes $\frac{1}{4}\exp\bigbrac{- (\log(1/\eps) )^{C_{k}/2}}$. Observe that $|\Pi_{\CB}f|\leq \Pi_{\CB}\nu$ and that the pseudorandomness condition~\eqref{pseudo-1} implies $\|\nu\|_{U^{2k}[N]} \leq 1+\exp \bigbrac{ -(\log(1/\eps))^{D_k}}$. Thus, as in~\eqref{eq:monotone}, we have
\begin{align*}
\|\Pi_{\CB}\nu\|_1, \|\nu\|_1\ll_k \|\nu\|_{U^{2k}[N]} \ll_k 1.   
\end{align*}
Consequently, by Definition~\ref{cond-ex}, we obtain
\begin{align}\label{l1-norm}
\norm{|f|+|\Pi_{\CB }f|}_1 \leq \norm{\nu}_1+ \norm{\Pi_{\CB}\nu}_1 \ll_k 1+ \frac{1}{N}\sum_{B \in \CB} \bigabs{\sum_{x \in B}\nu(x)} \ll_k 1+\norm{\nu}_1 \ll_k 1.
\end{align} 
Combining inequalities~\eqref{eq:fPiB} and~\eqref{l1-norm}, we have from the triangle inequality that for any $t \in [0,1/K)$,
\begin{align*}
	&\Bigabs{ \E_{n\in [N]}\bigbrac{ f  -\Pi_{\CB} f }(n) 1_{\Omega_\CB}(n)  \bigbrac{ F_{\mathcal{B}} (g_{\mathcal{B}} (n)\Gamma_{\mathcal{B}}) -t} }\\
	\geq& \Bigabs{ \E_{n\in [N]}  \bigbrac{ f  -\Pi_{\CB} f }(n) 1_{\Omega_\CB}(n) F_{\mathcal{B}} (g_{\mathcal{B}} (n)\Gamma_{\mathcal{B}})} -\frac{\norm{|f|+|\Pi_{\CB}f|}_1}{K}\\
	\geq& \frac{1}{5}\exp\bigbrac{- (\log(1/\eps))^{C_{k}/2}}
\end{align*}
if $C_k$ is chosen large enough so that the implied constant in~\eqref{l1-norm} can be absorbed.

In light of Lemma~\ref{lem-reg}, we may choose $t$ so that the induced factor $\CB_{h_{\mathcal{B}},K}$ with $h_{\mathcal{B}}(n)\coloneqq F_{\mathcal{B}}(g_{\mathcal{B}}(n)\Gamma_{\mathcal{B}})-t$ is $C$-regular and so that 
\begin{align}\label{eq:fPi_lower}
\bigabs{ \E_{n\in [N]} \bigbrac{ f  -\Pi_{\CB} f }(n) 1_{\Omega_\CB}(n) h_{\mathcal{B}} (n) }\geq \frac{1}{5} \exp\bigbrac{- (\log(1/\eps))^{C_{k}/2}}.
\end{align}

We then define $\CB' = \CB \vee \CB_{h_{\mathcal{B}}, K}$. This factor is $C$-regular and has complexity at most $(d+1+\frac{1}{2}\exp((\log(1/\varepsilon))^{C_k}),1)$ and resolution $K$.

It follows from the decomposition $[N]=\bigsqcup_{B \in \CB'} B$, Definition~\ref{cond-ex}, estimate~\eqref{l1-norm} and the fact that $|h_{\mathcal{B}}(n)-h_{\mathcal{B}}(n_B)|\leq \frac{1}{K}$ whenever $n,n_B$ are in the same atom $B\in \mathcal{B}'$, that
\begin{align*}
&\bigabs{ \E_{n\in [N]}\bigbrac{ f  -\Pi_{\CB} f }(n) 1_{\Omega_\CB}(n) h_{\mathcal{B}} (n) }\\
=& \Bigabs{\frac{1}{N}\sum_{B\in \CB'} \sum_{n\in B} \bigbrac{ f  -\Pi_{\CB} f }(n) 1_{\Omega_\CB}(n)  h_{\mathcal{B}}(n)}\\
\leq& \Bigabs{\frac{1}{N}\sum_{B\in \CB'} h_{\mathcal{B}}(n_B)\sum_{n\in B} \bigbrac{ f  -\Pi_{\CB} f }(n) 1_{\Omega_\CB}(n) } \\
&\quad+  \frac{1}{N}\sum_{B\in \CB'} \sum_{n\in B} \bigabs{(f-\Pi_{\CB }f)(n) }\cdot \bigabs{h_{\mathcal{B}} (n) -h_{\mathcal{B}}(n_B)} \\
   \leq& 2\cdot \frac{1}{N}\sum_{B\in \CB'} \bigabs{ \sum_{n\in B} (f -\Pi_{\CB} f) (n) } +\frac{\norm{|f|+|\Pi_{\CB }f|}_1}{K}\\
\leq& \frac{2}{N}\sum_{B\in \CB'} \sum_{n\in B} \bigabs{\Pi_{\CB'}f (n) -\Pi_{\CB}f(n) } +O_k(K^{-1}),
\end{align*}
since $\|h_\CB\|_{\infty} \le 1$, $\Pi_{\CB} f$ is constant on each atom of $\CB'$, and $\Omega_{\CB}$ is a union of atoms of $\CB$ (thus a union of atoms of $\CB'$).

Combining this with the lower bound~\eqref{eq:fPi_lower} and applying the Cauchy--Schwarz inequality gives us
\[
\| \Pi_{\CB'} f - \Pi_{\CB} f\|_2^2 \geq \frac{1}{101} \exp (-(\log(1/\eps) )^{C_{k}}),
\]
so the key claim follows.

\textbf{Iteration step.} We then iterate the key claim to prove the lemma. We define the initial factor $\CB^{(0)} = \set{[N]}$. With the notation of Definition~\ref{factor-res}, we have $\mathcal{B}^{(0)}=\mathcal{B}_{1/10,2}$. This factor has complexity $(1,1)$ and is $C$-regular for any large absolute constant $C$. 

We may assume that
\begin{align*}
\norm{ f-\Pi_{\CB^{(0)}} f}_{U^k[N]} >\eps	
\end{align*}
since otherwise we are done. Hence, by the key claim there exists a refinement $\mathcal{B}^{(1)}$ of $\mathcal{B}^{(0)}$ such that
\[
\norm{ \Pi_{\CB^{(1)}}f - \Pi_{\CB^{(0)}}f }_2^2  \geq \exp\bigbrac{- (\log(1/\eps))^{C_{k}}}/101
\]
and
\[
\mathcal{B}^{(1)}=\mathcal{B}^{(0)}\vee \mathcal{B}_{h_1,K}
\]
for some nilsequence $h_1$ of degree $k-1$, dimension $\leq (\log(1/\eps))^{C_k}$ and complexity at most $(\exp((\log(1/\eps))^{C_k})/2,1)$. Moreover, $\CB^{(1)}$ is $C$-regular.

Suppose then that for some integer $j\geq 1$, we have defined factors $\mathcal{B}^{(i)}$ of degree $k-1$, complexity at most $(i+\exp((\log(1/\eps))^{C_k})/2,1)$ and resolution at most $K$ for all $0\leq i\leq j$, such that $\mathcal{B}^{(i+1)}$ refines $\mathcal{B}^{(i)}$ for all $0\leq i\leq j-1$. Suppose additionally that
\[
\norm{ \Pi_{\CB^{(i+1)}}f - \Pi_{\CB^{(i)}}f }_2^2  \geq  \exp\bigbrac{- (\log(1/\eps))^{C_{k}}}/101
\]
for all $0\leq i\leq j-1$. If we have some $j\leq 200\exp((\log(1/\eps))^{C_k})$ and a corresponding set $\Omega_{\CB^{(j)}}$ which is a union of atoms of $\CB^{(j)}$, satisfying properties~\eqref{excep-ass-3.7}--\eqref{1-bdd-assump} and such that
\begin{align} \label{assump-tr2}
\norm{ (f-\Pi_{\CB^{(j)}} f)\cdot1_{\Omega_{\CB^{(j)}}} }_{U^k[N]} \leq \eps,	
\end{align}
we are done.

Assuming that this does not hold, from the key claim we obtain a refinement $\mathcal{B}^{(j+1)}$ of $\mathcal{B}^{(j)}$ of complexity at most $(j+1+\exp((\log(1/\eps))^{C_k})/2,1)$ and resolution at most $K$ such that $\mathcal{B}^{(j+1)}$ is $C$-regular and 
\[
\norm{ \Pi_{\CB^{(j+1)}}f - \Pi_{\CB^{(j)}}f }_2^2  \geq \exp\bigbrac{- (\log(1/\eps))^{C_{k}}}/101.
\]

We then iterate this for $J=1100\lceil \exp((\log(1/\eps))^{C_k})\rceil$ steps. If~\eqref{assump-tr2} fails for all $0\leq j\leq J$, we have
\begin{align}\label{eq:Pi_lowerbound}
\norm{ \Pi_{\CB^{(i+1)}}f - \Pi_{\CB^{(i)}}f }_2^2  \geq  \exp\bigbrac{- (\log(1/\eps))^{C_{k}}}/101
\end{align}
for all $0\leq i\leq J-1$. Now we may apply Lemma~\ref{pythagoras} and~\eqref{eq:Pi_lowerbound} to obtain for any $1\leq i\leq J$ the estimate
\begin{align*}
\| \Pi_{\CB^{(i)}} f \|_2^2 & = \|  \Pi_{\CB^{(i-1)}} f\|_2^2 + \| \Pi_{\CB^{(i)}} f - \Pi_{\CB^{(i-1)}} f\|_2^2\\
& \geq \|  \Pi_{\CB^{(i-1)}} f\|_2^2+  \exp\bigbrac{- (\log(1/\eps) )^{C_{k}}}/101.
\end{align*}
Iterating this, we obtain
\begin{align*}
\| \Pi_{\CB^{(J)}} f \|_2^2 \geq J \exp\bigbrac{- (\log(1/\eps))^{C_{k}}}/101.
\end{align*}

Note from the construction of $\CB^{(J)}$ that it is a $C$-regular factor of degree $k-1$, with complexity at most $(J+ \exp((\log(1/\eps))^{C_{k}})/2,1)$ and resolution at most $K$. Hence, by assumption~\eqref{l2-bd}, we have 
\[
\norm{\Pi_{\CB^{(J)}} f}^2_2 \leq 10.
\]
This gives us a contradiction since $J \exp\bigbrac{- (\log(1/\eps) )^{C_{k}}}>1010$. Hence, the iteration must terminate before step $J$, so~\eqref{assump-tr2} holds for some $0\leq j\leq J-1$. This completes the proof of the lemma. 
\end{proof}

\begin{remark}
We note that our energy increment argument differs somewhat from the approach of Green and Tao in \cite[Section 2]{GT-regularity} where they run the energy increment argument with measurable sets, establishing first that $f$ must correlate with a measurable set rather than a nilsequence. The advantage of this is that refining a factor $\CB$ of cardinality $M_0$ with the partition $\{E,[N]\setminus E\}$ produces at most $2M_0$ atoms, since each atom $B \in \CB$ can be refined into the atoms $B\cap E$ and $B\setminus E$. In contrast, a factor $\CB_{h,K}$ of resolution $K$ contains at most $K$ atoms. Refining each $B \in \CB$ using $\CB_{h,K}$ may therefore produce up to $K$ sets, resulting in at most $KM_0$ atoms in total. The first approach is thus more efficient. However, after iterating roughly $K$ successive refinements starting from the trivial factor $\{[N]\}$, the two approaches become essentially comparable, exhibiting essentially exponential growth in the number of atoms in terms of the number of iterations. Therefore we have chosen to use a direct nilsequence-based refinement.
\end{remark}


\subsection{Transferred regularity lemma II}

Our next goal is to simplify assumptions~\eqref{l2-bd}--\eqref{1-bdd-assump}. These assert that
\[
\|\Pi_{\CB} f  \|^2_2 \leq 10
\]
if $\CB$ is a $C$-regular factor of degree $k-1$, dimension at most $(\log(1/\eps))^{D_k/10}$, complexity at most $(\exp ((\log(1/\eps))^{D_k/10}), 1)$ and resolution at most $\exp ((\log(1/\eps))^{D_k/10})$, and that there exists a set $\Omega_\CB$ which is a union of atoms of $\CB$ and such that
\[
 |\Pi_\CB f \cdot 1_{\Omega_\CB}| \leq 2,\quad \text{and }\quad \Bigabs{\E_{n\in [N]} (|f|+1)(n) 1_{[N]\backslash \Omega_\CB}(n)} \leq \exp(-(\log(1/\eps))^{D_k}).
\]
By Definition~\ref{cond-ex}, the projection $\Pi_\CB f$ is defined as the conditional expectation of $f$ on the atoms $B \in \CB$. However, this projection is not directly computable, since the structure of the atoms $B$ is generally unknown. We will show that if a factor $\CB$ is regular, then each atom $B \in \CB$ can be well approximated by nilsequences. This will provide a more tractable way of estimating $\norm{\Pi_\CB f}_2$.

The following notation is adapted from \cite[Definition~2.2]{GT-regularity}. In contrast to that definition, we avoid introducing a growth function and instead track the complexity explicitly, as this plays an essential role in our argument. Furthermore, for technical reasons, we require approximation in the $\ell^p$-norm for arbitrarily large $p$, rather than in the $\ell^2$-norm.

\begin{definition}[Measurability]\label{meas}
Let $N,M,\widetilde{C}\ge 1$ and let $s$ be a natural number.  
A finite set $E\subseteq [N]$ is said to be \emph{$s$-measurable with complexity $(M,\widetilde{C})$} if, for every $p\ge 1$, there exists a $1$-bounded, real-valued nilsequence $\psi$ of degree at most $s$, dimension at most $M$, and complexity at most $M^{10p^2+1}$, such that
\[
\|1_E - \psi\|_{p}^p \le \widetilde{C} M^{-10p^2}.
\]
\end{definition}

\begin{lemma}[Atoms of regular factors are measurable]\label{lem-meas}
Let $C\geq 1$ be fixed. Let $k \in \mathbb{N}$, and let $2\leq K\leq M$ be large enough parameters depending on $k$. Let $N\geq 1$, and suppose that $h = F(g(\cdot)\Gamma)$ is a real-valued nilsequence on $[N]$ of degree $k-1$, dimension at most $M$ and complexity at most $(M,1)$. Let $\CB_{h,K}$ be a $C$-regular factor induced by $h$ and of resolution $K$. Then each atom $B \in \CB_{h,K}$ is $(k-1)$-measurable with complexity $(M,4C)$.
\end{lemma}

\begin{proof}
Fix an integer $j \in \mathbb{Z}$ and consider the set
\[
B=\set{ n\in [N]\colon  \frac{j}{K} < F(g(n)\Gamma) \leq \frac{j+1}{K}},
\]
which is an atom of the factor $\CB_{h,K}$ by Definition~\ref{factor-res}. Let $\eps >0$ be a small parameter to be chosen later. Define a cutoff $\eta\colon \R\to[0,1]$ such that $\eta(x)=0$ for $x \notin [\frac{j-\varepsilon}{K}, \frac{j+1+\varepsilon}{K})$ and $\eta(x)=1$ for $x \in [\frac{j}{K}, \frac{j+1}{K})$, with
\[
\|\eta\|_{\lip} \leq 1 + K/\varepsilon,
\]
which is achievable by taking $\eta$ to be a piecewise linear function.

By the definitions of the set $B$ and function $\eta$, we have for any real number $p\geq 1$ the estimate 
\begin{align}\label{eq:1Bh}\begin{split}
\norm{1_B - \eta\circ h}_p^p & =\E_{n\in [N]} \bigabs{1_B(n) -\eta (F(g(n)\Gamma))}^p \\
&\leq \E_{n\in [N] } 1_{F(g(n)\Gamma) \in [\frac{j-\eps }{K},\frac{j}{K})\cup [\frac{j+1 }{K},\frac{j+1+\eps }{K})}
|\eta(F(g(n)\Gamma))|^p.
\end{split}
\end{align}
Since $\CB_{h,K}$ is $C$-regular, for any $r>0$ we have from Definition~\ref{reg-fac} that
\[
\#\set{n\in [N]\colon  \norm{K\cdot F(g(n)\Gamma)}_\T \leq r}\leq 2CrN.
\]
Choosing $r=\varepsilon$, it follows that
\[
\#\{n\in [N]\colon |F(g(n)\Gamma)- j/K|\le \varepsilon / K \}
   \leq 2C \eps N,
\]
uniformly for all $j\in \Z$.  
Since $|\eta|\le 1$, this and~\eqref{eq:1Bh} imply
\[
\|1_B - \eta\circ h\|_{p}^{p} \leq 2C\varepsilon.
\]

Now we choose $\varepsilon = 2 M^{-10p^{2}}$ (note that $\varepsilon <1$ because $K\le M$ and $M\geq 2$ and $p\ge 1$).  
With this choice we obtain
\[
\|\eta\|_{\lip} \leq KM^{10p^{2}}/2 + 1 \leq M^{10p^2+1}/2+1,
\qquad\text{and}\qquad
\|1_B - \eta\circ h\|_{p}^p \leq 4CM^{-10p^2}.
\]
The composition $\eta\circ h = \eta\circ F (g(\cdot)\Gamma)$ is a nilsequence on the same polynomial orbit $n\mapsto g(n)\Gamma$ (of degree $k-1$), and its Lipschitz norm satisfies
\begin{align*}
\|\eta\circ F\|_{\lip}
&\le \|\eta\circ F\|_\infty
   + \sup_{x\ne y} \frac{|\eta(F(x)) - \eta(F(y))|}{d(x,y)} \\
&\leq 1
   + \sup_{\substack{x\ne y \\ F(x)\ne F(y)}}
      \frac{|\eta(F(x)) - \eta(F(y))|}{|F(x)-F(y)|}
      \cdot \frac{|F(x)-F(y)|}{d(x,y)}  \\
&\leq 2+M^{10p^{2}+1}/2\\
&\leq M^{10p^2+1}.
\end{align*}
Combining the two inequalities above and invoking Definition~\ref{meas}, we obtain the desired conclusion of the lemma.
\end{proof}

\begin{lemma}[Approximating measurable sets by nilsequences]\label{set-nil}
Let $N\geq 1$, and let $C\geq 1$ be fixed. Let $s,m\in \mathbb{N}$, and let $M\geq 1$ be a large number. Suppose that $\CB_1,\dots,\CB_m$ are factors of $[N]$ whose atoms are $s$-measurable sets of complexity at most $(M,C)$. Let $\CB =\bigvee_{1\leq j\leq m}\CB_j$. Then for any atom $B \in \CB$ there exists a real-valued nilsequence $\psi_B$ of degree $s$, dimension at most $mM$ and complexity at most $mM^{11m^3}$, such that  
\[
\E_{n\in [N]}\bigabs{1_B(n) -\psi_B (n)} \leq \frac{Cm}{M^{10m}}.
\]
\end{lemma}

\begin{proof} 
For an atom $B\in\CB$, write $B = \bigcap_{j = 1}^m B_j$ with $B_j \in \CB_j$. Then we have $1_B(n) = \prod_{1\leq j\leq m} 1_{B_j} (n)$.

By the $s$-measurability of each $B_j$, we may take $p=m$ in Definition~\ref{meas} to obtain $1$-bounded, real-valued nilsequences $\psi_1,\dots,\psi_m$ of degree $s$, dimension at most $M$, and complexity at most $M^{10m^2+1}$ such that
\begin{align}\label{lm-app}
	\norm{ 1_{B_j} -\psi_j}_m \leq C M^{-10m}.
\end{align}
Moreover, using a telescoping identity, we have
\begin{align*}\begin{split}
	&\E_{n\in [N]} \Bigabs{ \prod_{1\leq j\leq m} 1_{B_j} (n)-
	 \prod_{1\leq j\leq m} \psi_j (n)} \\ 
	\leq& \sum_{l=1}^m \E_{n\in [N]} \biggabs{ \brac{\prod_{j<l} 1_{B_j} (n)} (1_{B_l}-\psi_l)(n) \brac{\prod_{l<j\leq m} \psi_j(n)}}\\
	\leq& m \max_{1\leq l\leq m}  \E_{n\in [N]} \biggabs{ \brac{\prod_{j<l} 1_{B_j} (n)} (1_{B_l}-\psi_l)(n) \brac{\prod_{l<j\leq m} \psi_j(n)} }.
    \end{split}
\end{align*}
Let $g_j\in \{1_{B_j},\psi_j\}$ for each $j\neq l$, thus $|g_j|\leq1$. Applying H\"older's inequality and~\eqref{lm-app}, we can estimate
\[
\E_{n\in [N]} \Bigabs{  (1_{B_l} -\psi_l) (n)\prod_{\substack{1\leq j\leq m \\ j\neq l} } g_j(n)}\leq \norm{1_{B_l} -\psi_l}_m \prod_{\substack{j\leq m \\ j\neq l}} \norm{g_j}_{m}\leq  CM^{-10m}.
\]
Combining the two inequalities above, recalling that $1_B =\prod_{1\leq j\leq m} 1_{B_j}$, we thus have
\[
\E_{n\in [N]} \Bigabs{1_B(n) - \prod_{1\leq j\leq m}\psi_j (n)} \leq \frac{Cm}{M^{10m}}.
\]

It remains to verify that $\psi_B \coloneqq \prod_{j=1}^m \psi_j$ is a nilsequence with the required degree, dimension, and complexity bounds. This follows immediately from Lemma~\ref{pro-nil}.
\end{proof}

With these preparatory steps in place, we are now ready to simplify conditions~\eqref{l2-bd}--\eqref{1-bdd-assump}.

\begin{lemma}\label{lem-cor}
Let $C\geq 1$ be fixed. Let $1\leq K\leq M$ be large enough in terms of $C$, and let $N\geq 1$. 
Let $k\geq 2$ be a natural number and let $\CB$ be a $C$-regular factor of degree $k-1$, complexity at most $(M,1)$ and resolution $K$.
Let $M^{-4M} \leq  \eta \leq M^{-2M}$. Suppose that $f \colon  [N] \to \mathbb{R}_{\geq 0}$ is a function and that there exists a function $\tilde f\colon [N]\to\R_{\ge 0}$ with the following three properties:
\begin{enumerate}[(i)]
    \item We have $f\le \tilde f$ pointwise; 
    \item $\tilde f$ obeys the pointwise bound $\|\tilde f\|_\infty \leq \eta^{-1/3}$;
    \item For every nilsequence $F(g(\cdot)\Gamma)$ of degree $k-1$, dimension at most $CM^2$ and complexity at most $M^{CM^3}$ there is a partition $\mathcal P$ of $[N]$ with
\begin{align}\label{corr-de}
\left|\sum_{n\in [N]} \tilde f(n) F(g(n)\Gamma)\right| \leq \frac{\eta N}{4} +\sum_{P\in \mathcal P}\left|\sum_{n\in P} F(g(n)\Gamma)\right|.
\end{align}
\end{enumerate}
Then we have
\begin{align}\label{l_2_bd}
\norm{\Pi_\CB f}_2^2 \leq 5,
\end{align}
and there exists a set $\Omega_\CB$ which is a union of atoms of $\CB$ and such that
\begin{align}\label{almost_all_bd}
 |\Pi_\CB f \cdot 1_{\Omega_\CB}| \leq 2,\quad \text{and }\quad \Bigabs{\E_{n\in [N]} ( f+1)(n) 1_{[N]\backslash \Omega_\CB}(n)} \leq \eta^{1/2}.
\end{align} 
\end{lemma}

\begin{proof}
We begin by showing that there exists an exceptional set $\mathcal{E}\subseteq [N]$ satisfying the following properties:
\begin{enumerate}
    \item $\mathcal{E}$ is a union of atoms of $\CB$.
	\item For every atom $B\in \CB$ with $B\cap \mathcal E=\emptyset$ we have $\E_{n \in [N]} (f+1)(n) 1_B(n) > \eta$. 
	\item The contribution from $\mathcal E$ is negligible, that is,
\begin{align} \label{excep-con}
\bigabs{\E_{n\in [N]} (f+1) (n) 1_\mathcal E(n)} \leq \eta^{1/2}.
\end{align}
\item We have
\begin{align}\label{inf-f-fact}
\norm{\Pi_\CB f\cdot (1-1_\mathcal E)}_\infty \leq  2.
\end{align}
\end{enumerate}

To construct such a set, we call an atom $B\in\CB$ \emph{small} if
\[
\E_{n\in [N]} (f+1) (n) 1_B(n) \leq \eta.
\]	
Let $\mathcal D\subseteq \CB$ be the collection of small atoms and set $\mathcal E = \bigcup_{B\in \mathcal D}B$. Then we have
\[
\E_{n\in [N]} (f+1) (n) 1_\mathcal E (n) = \sum_{B\in \mathcal D} \E_{n\in [N]} (f+1) (n) 1_B(n)\leq  \eta |\mathcal D| \leq \eta |\CB|.
\]
On the other hand, in light of Definition~\ref{factor-res} and the assumptions on $\CB$, one has $|\CB|\leq K^M$. We thus obtain from the condition $M^{-4M}\leq \eta\leq M^{-2M}$ that
\[
\E_{n\in [N]} (f+1) (n) 1_\mathcal E(n) \leq \eta K^M \leq \eta^{1/2}.
\]
Hence, the first three properties required of the set $\mathcal E$ are satisfied. It remains to verify property~\eqref{inf-f-fact}.

We now restrict our attention to atoms $B\not\in \mathcal D$, which means that 
\begin{align}\label{lar-b}
\E_{n\in [N]} f(n) 1_B(n) +\E_{n\in [N]}1_B(n) = \E_{n\in [N]} (f+1) (n) 1_B(n)>\eta.
\end{align}
By~\eqref{eq:conditionalexp}, for any $B \in \CB$, one has
\begin{align}\label{eq:PiB}
\E_{n\in B} \Pi_\CB f(n) = \E_{n\in B} f(n) =\frac{\E_{n\in [N]} f(n) 1_B(n)}{\E_{n\in [N]}1_B(n)} \leq \frac{\E_{n\in [N]} \tilde f(n) 1_B(n)}{\E_{n\in [N]}1_B(n)}.
\end{align}

Since $\CB$ is $C$-regular for some absolute constant $C\geq 1$ and has degree $k-1$, complexity $(M,1)$ and resolution $K$, it follows from Definition~\ref{reg-fac} that there are $C$-regular factors $\CB_{h_1,K},\dots,\CB_{h_m,K}$ with $m\leq M$ and $h_i$ being real-valued nilsequences of degree $k-1$, dimension at most $M$ and complexity at most $(M,1)$ such that $\CB=\bigvee_{i\leq m}\CB_{h_i,K}$. As Lemma~\ref{lem-meas} ensures that the atoms of $\CB_{h_i,K}$ are $(k-1)$-measurable with complexity at most $(M,4C)$, one can deduce from Lemma~\ref{set-nil} with the assumptions $\|\tilde f\|_\infty \leq\eta^{-1/3}$ and $M^{-4M}\leq\eta \leq M^{-2M}$ that $\frac{CM\|\tilde f\|_\infty}{M^{10M}}\leq \eta^{3/2}$, and that
\[
\Bigabs{\E_{n\in [N]} \tilde f(n)1_B(n) -\E_{n\in [N]} \tilde f(n) F(g(n)\Gamma)}\leq \|\tilde f\|_\infty \E_{n\in [N]}\bigabs{1_B(n) -F(g(n)\Gamma)}\ll\eta^{3/2}
\]
for some nilsequence $F(g(\cdot)\Gamma)$ of degree $k-1$, dimension at most $CM^2$ and complexity at most $M^{CM^3}$. Since $M$ is large enough and $\eta\leq M^{-2M}$, we may assume that the error term in the previous equation is at most $\eta/100$ in modulus. Now, in light of assumption~\eqref{corr-de}, one has
\[
\Bigabs{\sum_{n\in [N]}\tilde  f(n)1_B(n)} \leq \sum_{P\in \mathcal P} \Bigabs{\sum_{n\in P} F(g(n)\Gamma)} +(\eta/4+\eta/100)N,
\]
for some partition $\mathcal P$ of $[N]$. Noting that $M^{-10M+1}\leq \eta^{3/2}$ we then deduce from Lemma~\ref{set-nil} that
\[
\sum_{P\in\mathcal P} \sum_{n\in P}\bigabs{ F(g(n)\Gamma)-1_B(n)}\ll \eta ^{3/2}N.
\]

Combining the above two inequalities gives us
\[
\E_{n\in[N]} \tilde f(n) 1_B(n) \leq \eta/3	+ \E_{n\in [N]}1_B(n).
\]
On the other hand, combining this inequality with~\eqref{lar-b} we obtain
\[
\E_{n\in [N]}1_B(n)\geq \eta/3.
\]
Thus, recalling~\eqref{eq:PiB}, we can bound the average of $\Pi_\CB f$ on $B$ by
\[
\E_{n\in B} \Pi_\CB f(n)
\leq \frac{\eta/3 + \E_{n\in [N]}1_B(n)}{\E_{n\in [N]}1_B(n)}
\leq 2
\]
for all atoms $B \subseteq [N]\setminus \mathcal E$. Since $\Pi_\CB f$ is constant on each atom, we obtain $\sup_{B\cap\mathcal E=\emptyset}\sup_{n\in B}|\Pi_{\CB}f(n)|\leq 2$, giving~\eqref{inf-f-fact}. 

We are now ready to deduce the conclusion of the lemma. Setting $\Omega_\CB=[N]\backslash\mathcal E$, the estimates~\eqref{almost_all_bd} follow immediately from~\eqref{excep-con} and~\eqref{inf-f-fact}. It remains to prove~\eqref{l_2_bd}. Observe first that
\begin{align}\label{eq:Pi_split}
\norm{\Pi_\CB f}_2^2 = \frac{1}{N}\sum_{B\in \CB} \sum_{n\in B} \abs{\Pi_\CB f(n)}^2 = \frac{1}{N} \sum_{B\in \mathcal{B}\setminus \mathcal D} \sum_{n\in B} \abs{\Pi_\CB f(n)}^2 +\frac{1}{N}\sum_{B\in \mathcal D}\sum_{n\in B} \abs{\Pi_\CB f(n)}^2,
\end{align}
where the second equality follows from decomposing the average over $\CB$ into contributions from atoms inside and outside $\mathcal D$. Now, since $\mathcal E=\cup_{B\in \mathcal D}B$, it follows from~\eqref{inf-f-fact} that the first term on the right of~\eqref{eq:Pi_split} is bounded by $4$. Then, noting that $\Pi_{\mathcal{B}}f(n)=\mathbb{E}_{m\in B}f(m)$ for $n\in B$ and that $f\geq 0$, we have
\begin{align*}
\norm{\Pi_\CB f}_2^2 \leq 4+ \frac{1}{N} \sum_{B\in \mathcal D} \sum_{n\in B} |\E_{m\in B}f(m)|^2 \leq 4+ \frac{\norm{f}_\infty}{N}\sum_{B\in \mathcal D} \sum_{m\in B}  f(m).
\end{align*}
Invoking the bound $\|f\|_\infty \leq \eta^{-1/3}$, we deduce from~\eqref{excep-con} that
\[
\norm{\Pi_\CB f}_2^2\leq 4+ \norm{f}_\infty \cdot \E_{n \in [N]}   \Pi_\CB f(n)1_\mathcal E(n) \leq 4+ \eta^{-1/3}\eta^{1/2} \leq 5,
\]
since $\eta\leq M^{-2M}<1$.
\end{proof}

We summarise the above result as a \emph{dense model} statement. Compared with~\cite{Go-dense} and~\cite{reing}, we obtain quasipolynomial dependencies (improving on previous exponential bounds). We also note that the result below provides additional structural information on the dense function $g$, which can be viewed as the conditional expectation of the function with respect to a suitable regular factor.

\begin{proposition}[Dense model]
\label{tra-wea-2}
Let $N \geq 1$ be an integer, let $k \geq 2$ be a natural number, and let $0 < \varepsilon < 1/3$ be a parameter. Let $f \colon [N] \to \mathbb{R}_{\geq 0}$ be a function. Then there exists a small number $0<\gamma_k<1/10$ depending only on $k$ such that the following statement holds.

Suppose there exist two majorant functions $\nu_1,\nu_2 \colon [N] \to \mathbb{R}_{\ge 0}$\footnote{If there exists a single majorant function $\nu$ satisfying all three hypotheses, then one may simply take $\nu=\nu_1=\nu_2$. We formulate the assumptions using two possibly different majorants in order to weaken the hypotheses.} such that $f \le \nu_i$ pointwise for each $i \in \{1,2\}$, and suppose that the following conditions hold:
\begin{enumerate}
    \item (Pseudorandomness) 
    \[
    \|\nu_1 - 1 \|_{U^{2k}[N]} \leq \varepsilon.
    \]

    \item (Boundedness of majorant) 
    \[
    \| \nu_2 \|_\infty \leq \exp\exp((\log(1/\varepsilon))^{\gamma_k^2}).
    \]
    
    \item (Nilsequence correlation) Let $\exp\bigbrac{- \exp\bigbrac{(\log(1/\eps) )^{1/3}}}\leq \eta<1$. Then, for every nilsequence $F(g(\cdot)\Gamma)$ of degree $k-1$, dimension at most $(\log(1/\eta))^3$, and complexity at most $(\exp((\log(1/\eta))^3),\exp((\log(1/\eta))^3))$, there exists a partition $\mathcal{P}$ of $[N]$ into arithmetic progressions with common difference at most $\exp((\log(1/\eta))^{C_k})$ for some constant $C_k$ depending only on $k$, such that
    \begin{align}\label{corr-nil-f}
    \left|\sum_{n\in [N]} \nu_2(n) F(g(n)\Gamma)\right|\leq \eta^2N +\sum_{P\in \mathcal{P}}\left|\sum_{n\in P} F(g(n)\Gamma)\right|.
    \end{align}
\end{enumerate}

\bigskip

Then there exists a bounded function $g\colon[N]\to[0,2]$ such that
\[
\norm{f-g}_{U^k[N]} \ll_k \exp(-(\log(1/\varepsilon))^{\gamma_k}).
\]

Additionally, there exist a factor $\CB$ of complexity at most 
$\exp\exp\bigl((\log(1/\eps))^{\gamma_k}\bigr)$ and a set $\Omega\subseteq [N]$, such that
\[
g(n) =\Pi_{\CB} f(n) \qquad \text{ for all } n\in \Omega,
\]
and 
\[
\bigabs{\E_{n\in [N]} (f+1)(n) 1_{[N]\backslash \Omega}(n)} \leq \eps.
\]
\end{proposition}

\begin{proof}
Let $C\geq 1$ be an absolute constant. We first claim that the inequality
\[
\norm{\Pi_\CB f }^2_2 \leq 5
\]
holds for every $C$-regular factor $\CB$ of degree $k-1$, dimension $d\leq (\log(1/\eps))^{1/10}$, complexity at most $(M,M)$ with
$M \leq \exp\bigbrac{(\log(1/\eps))^{1/10}}$, and resolution at most $\exp\bigbrac{(\log(1/\eps))^{1/10}}$. Moreover, for any sufficiently small constant $0<\gamma_k<1/10$ (to be chosen later), it is harmless to assume the lower bound
\[
M\geq \exp((\log(1/\eps))^{\gamma_k^2}).
\]

To prove the claim, set $1/\eta =M^{3M}$. Then $1/\eta< \exp\exp ((\log(1/\eps))^{1/3})$, so this choice of $\eta$ satisfies condition (3). Moreover, note that $M^{CM^3} \leq \exp((\log (1/\eta))^3) $ and $M^2 \leq (\log(1/\eta))^3$.
Consequently, for $M$ sufficiently large, the hypothesis~\eqref{corr-nil-f} implies that~\eqref{corr-de} holds for every nilsequence of degree $k-1$, dimension at most $CM^2$ and complexity at most $M^{CM^3}$. 
On the other hand, since $\gamma_k<1/10$ is small, the boundedness assumption (2) yields
\[
\norm{\nu_2}_\infty \leq \exp\exp ((\log(1/\eps))^{\gamma_k^2})<\eta^{-1/3}.
\]
Hence the second condition of Lemma~\ref{lem-cor} also holds (with $\tilde f=\nu_2$). Therefore, Lemma~\ref{lem-cor} implies that for every such $C$-regular factor $\CB$,
\[
\norm{\Pi_\CB f }^2_2 \leq 5.
\]
Moreover, Lemma~\ref{lem-cor} provides a corresponding set $\Omega_\CB$, which is a union of atoms of $\CB$ and satisfies the cardinality bound~\eqref{almost_all_bd}. Finally, recall that $1/\eta =M^{3M}$ and $M \geq \exp\bigbrac{(\log(1/\eps))^{\gamma_k^2}}$. It follows that for any sufficiently small constant $\gamma_k>0$ we always have $\eta^{1/2} \leq \eps $. Therefore, we may rewrite~\eqref{almost_all_bd} in the form
\begin{align*}
|\Pi_\CB f1_{\Omega_\CB}| \leq 2 
\qquad \text{and}\qquad 
\bigabs{\E_{n\in [N]} (f+1)(n)1_{[N]\backslash\Omega_\CB}(n)} \leq \eps.	
\end{align*}

Since $0\leq f\leq \nu_1$ and $\norm{\nu_1-1}_{U^{2k}[N]} \leq \eps$, and since the preceding bounds verify~\eqref{l2-bd}--\eqref{1-bdd-assump},
we may apply Lemma~\ref{tra-wea} with $\exp (-(\log (1/\eps ) )^{D_k})$ in place of $\eps$, and then apply Lemma~\ref{lem-meas}. This yields a constant $0<\gamma_k<1/10$ (depending only on $k$),
and a sequence of $C$-regular factors $\CB_1,\dots,\CB_T$ with $T \ll_k \exp\!\bigl((\log(1/\eps))^{\gamma_k}\bigr)$, whose atoms are $k-1$-measurable with complexity at most $M\ll_k \exp\!\bigl((\log(1/\eps))^{\gamma_k}\bigr)$, 
and such that $|\CB_j|\ll K\ll \exp\!\bigl((\log(1/\eps))^{\gamma_k}\bigr)$, and
\begin{align}\label{uni-bound}
\norm{(f-\Pi_{\vee_{1\leq i\leq T}\CB_i} f)\cdot 1_{\Omega_T} }_{U^k[N]} 
\ll_k \exp (-(\log(1/\eps))^{\gamma_k}),
\end{align}
where $\Omega_T$ is the exceptional set produced by Lemma~\ref{tra-wea} and satisfies
\begin{align}\label{**}
\bigabs{\E_{n\in [N]} (f+1)(n) 1_{[N]\backslash \Omega_T}(n)} \leq \eps.
\end{align}
Let $\CB\coloneqq \bigvee_{1\leq i\leq T}\CB_i$ and $\Omega\coloneqq \Omega_T$. Define a bounded function $g\colon [N]\to[0,2]$ by setting
\[
g(n)\coloneqq \Pi_\CB f(n)\cdot 1_{\Omega}(n).
\]
It then follows from~\eqref{uni-bound} and the triangle inequality for the Gowers norm that
\begin{align*}
	\norm{f-g}_{U^k[N]} &\ll_k \norm{(f-g)\cdot 1_\Omega}_{U^k[N]} + \norm{f\cdot 1_{[N]\backslash\Omega}}_{U^k[N]}\\
	&\ll_k \exp (-(\log(1/\eps))^{\gamma_k})+ \norm{f\cdot 1_{[N]\backslash\Omega}}_{U^k[N]}.
\end{align*}
We claim that $\norm{f\cdot 1_{[N]\backslash\Omega}}_{U^k[N]}\ll_k \exp (-(\log(1/\eps))^{\gamma_k})$. Indeed, if this were false, then by Theorem~\ref{tra-inv} (and the choice of
 $\gamma_k$) there would exist a $1$-bounded function $\psi$ such that
\[
\bigabs{\E_{n\in [N]} f(n) 1_{[N]\backslash\Omega}(n) \psi(n)} \geq \eps^{1/10},
\]
However this contradicts~\eqref{**}. Therefore, we have
 \[
\norm{f-g}_{U^k[N]} \ll_k \exp (-(\log(1/\eps))^{\gamma_k}).
\]
This completes the proof.
\end{proof}

\begin{remark}
We will apply this Proposition~\ref{tra-wea-2} with $\nu_2=\Lambda(W\cdot +b)$; 
this is possible since we have a strong bound on the $U^k[N]$-norm of the von Mangoldt function. Note however that we could also take $\nu_2$ to be a sieve majorant for $\Lambda(W\cdot+b)$, and in situations where we do not control the $U^k[N]$-norm of the function directly, this would be necessary.	We take $\nu_1$ in turn to be a GPY sieve as in earlier works; it would in fact be possible to take $\nu_1=\nu_2$, but since we later need a quantitative linear forms condition for $\nu_1$, this would require a slight additional argument, whereas for the GPY sieve we get the linear forms condition directly from existing work. 
\end{remark}
\section{Correlation of the von Mangoldt function with nilsequences}\label{sec4}

In this section, we verify condition~\eqref{corr-nil-f} with the choice $\nu_2=\Lambda(W\cdot+b)$ as the $W$-tricked version of the von Mangoldt function. The key property of the von Mangoldt function used in the proof is that it decomposes into Type~I and Type~II sums, which we define as follows.

\begin{definition}[Type I and Type II sums]
A function $a\colon \Z \to \C$ is called \emph{divisor-bounded} if $|a(n)| \leq (d(n)\log(3n))^{100}$ for all $n$.

Let $N\geq 1$ and $U\geq 1$ be parameters. We say that a function $f\colon \mathbb{Z}\to \mathbb{C}$ is a \emph{Type~I sum of level $U$ at scale $N$} if for $n\in [N]$ we can write
\[
f(n) = \sum_{d \mid n} a(d)
\]
for some divisor-bounded function $a$ with $\supp(a)\subseteq [1,U]$.

For an interval $J$, we say that a function $f\colon \mathbb{Z}\to \mathbb{C}$ is a \emph{Type~II sum of range $J$ at scale $N$} if there is $U\in J$ such that for $n\in [N]$ we can write
\[
f(n) = \sum_{n=md} a(d)b(m)
\]
for some divisor-bounded functions $a,b$, with $\supp(a) \subseteq [U, 2U]$.
\end{definition}

We wish to apply the dense model statement (Proposition~\ref{tra-wea-2}) to a $W$-tricked version of the von Mangoldt function. This involves verifying
the correlation condition~\eqref{corr-nil-f} between nilsequences and the von Mangoldt function. To this end, we decompose the relevant polynomial sequence into rational and smooth polynomial sequences, which are easier to analyse.

\begin{definition}[Rational and smooth sequences]\label{rational-smooth}
Let $G/\Gamma$ be a nilmanifold with a Mal'cev basis $\mathcal X$. Let $1\leq Q,M\leq N$. We say that $\gamma\in G$ is $Q$-\emph{rational} if $\gamma^q\in \Gamma$ for some integer $1\leq q\leq Q$. We say that a sequence $(\gamma(n))_{n\in \Z}$ is $Q$-\emph{rational} if $\gamma(n)$ is $Q$-rational for every $n\in \Z$.

We say that a polynomial sequence $(\eps(n))_{n\in \Z}$ is $(M,N)$-\emph{smooth} if
\[
d\bigbrac{\eps(n), \textnormal{id}_G} \leq M \quad \textnormal{ and }\quad d\bigbrac{\eps(n), \eps(n+1)} \leq M/N
\]
for all $n\in [N]$.
\end{definition}

The goal of this section is to establish the following proposition.

\begin{proposition}[Correlation with nilsequences]\label{cor-1}
Let $k$ be a natural number, let $C\geq 1$ and $N\geq 3$. Let $1\leq b\leq W\leq \exp((\log N)^{1/200})$ with $(b,W)=1$. Let $f\colon [N] \to \R_{\geq 0}$ be a function with the following two properties:
\begin{enumerate}
	\item For $n\in [N]$, we can write
	\[
	f(n)=\sum_{0\leq j\leq (\log N)^{C}}\int_{1}^{N}\bigl(h^{\textnormal{I}}_{j,t}(n)+h^{\textnormal{II}}_{j,t}(n)\bigr)\frac{\rd t}{t},
	\]
	where for each $j$ and $t$, the function $h^{\textnormal{I}}_{j,t}$ is a Type~I sum of level $N^{2/3}$ at scale $N$, and the function $h^{\textnormal{II}}_{j,t}$ is a Type~II sum of range $[N^{1/3}/2, N^{2/3}]$ at scale $N$.
	\item For any arithmetic progression $P\subseteq[N]$ of size $\geq N \exp(-(\log N)^{1/100})$ we have $\E_{n\in P} f(n) \leq C$.
\end{enumerate}

Then there exists a large constant $C_k>1$ depending only on $k$, such that the following holds. If $\eta \in(0,1)$ satisfies $\exp\exp((\log \log N)^{1/C_k}) \leq 1/\eta$ and $(\log(1/\eta))^{200C_k^{10^k}} \leq \log N$, and if $F(g(\cdot)\Gamma)$ is a nilsequence of degree at most $k-1$, dimension at most $(\log(1/\eta))^{C_k}$ and complexity at most $(1/\eta,1/\eta)$, then there exists a partition $\mathcal P$ of $\set{n\in [N]\colon n\equiv b\mod W}$ into progressions, each of cardinality at least $(N/W)\exp(-5(\log(1/\eta))^{C_k^{10^k}})$, such that
\[
\Bigabs{\twosum{n\in [N]}{n\equiv b \mod W} f(n) F(g(\tfrac{n-b}{W})\Gamma)} \leq C\sum_{P\in \mathcal P}\Bigabs{\sum_{n\in P} F(g(\tfrac{n-b}{W})\Gamma)} +\eta^{2} N/W.
\]
\end{proposition}

\begin{proof}
We aim to prove the proposition by an induction on the step of the underlying nilmanifold, following the strategy in
\cite[Theorem~7]{Leng} and \cite[Proposition~4.6]{MTW}. Throughout, let $C_k>2$ be a sufficiently large constant depending only on $k$. Suppose first that $F(g(\cdot)\Gamma)$ is a nilsequence of degree $k-1$, dimension $d\leq (\log(1/\eta))^{C_k}$
and complexity at most $(1/\eta,1/\eta)$, and that it satisfies the correlation inequality
\[
\bigabs{\E_{n\in [N]\atop n\equiv b\mod W} f(n) F(g(\tfrac{n-b}{W})\Gamma)} \geq \eta^{C_k}.
\]

If no nilsequence obeys this inequality, then the desired conclusion is immediate; hence we may assume that such a nilsequence exists. We may normalise so that $\|F\|_{\mathrm{Lip}}\le 1$. Next, by \cite[Lemma~A.6]{Leng} together with \cite[Lemma~2.2]{Leng}, we may further assume that $F$ is a vertical character of frequency $\xi$ with $|\xi|\leq \eta^{-(2d)^{C_k}}$ and that the nilmanifold $G/\Gamma$ has a one-dimensional vertical component. This reduction comes at the cost of weakening the correlation bound to
\begin{align}\label{4.1}
\bigabs{\E_{n\in [N] \atop n\equiv b\mod W} f(n) F(g(\tfrac{n-b}{W})\Gamma)} \geq \exp\bigbrac{ -12(\log(1/\eta))^{C_k}}.
\end{align}

We now claim that for each integer $0\leq j\leq k-1$, there exists a partition $\mathcal P^{(j)}$ of $\set{n\in[N]\colon n\equiv b\mod W}$ into arithmetic progressions, each of cardinality $\geq (N/W)\exp\bigbrac{- 5(\log(1/\eta))^{C_k^{10^{j}}}}$ such that
\begin{align}\label{bias-induction}
\E_{P^{(j)} \in \mathcal P^{(j)}} \bigabs{\E_{n \in P^{(j)}} f(n) F_{P^{(j)}} (g_{P^{(j)}} (\tfrac{n-b}{W}) \Gamma)} \geq \exp\bigbrac{ -12(\log(1/\eta))^{C_k^{10^{j}}}},
\end{align}
and
\begin{multline}\label{4.3}
\Bigabs{\sum_{n\in [N] \atop n\equiv b\mod W} f(n) F(g(\tfrac{n-b}{W})\Gamma)} \leq
\sum_{P^{(j)} \in \mathcal P^{(j)}} \bigabs{\sum_{n \in P^{(j)}} f(n) F_{P^{(j)}} (g_{P^{(j)}} (\tfrac{n-b}{W}) \Gamma)} \\
+O\bigbrac{(N/W) \exp\bigbrac{ -\tfrac{(\log(1/\eta))^{C_k^{10}}}{j+1}}}\,,
\end{multline}
where $g_{P^{(j)}}$ is a polynomial sequence on a $\exp\bigbrac{(\log(1/\eta))^{C_k^{10^j}}}$-rational subgroup of $G$ whose step is at most $k-1-j$, and $F_{P^{(j)}}$ is a Lipschitz function such that
\begin{align}\label{ass-lip-varying}
\bigabs{ F(g(\tfrac{n-b}{W})\Gamma) - F_{P^{(j)}} (g_{P^{(j)}} (\tfrac{n-b}{W}) \Gamma)}\leq \exp\bigbrac{ -\tfrac{(\log(1/\eta))^{C_k^{10}}}{j+1}}
\end{align}
whenever $n\in P^{(j)}$.

When $j=0$, we take $\mathcal P^{(0)}=\{\set{n\in[N]\colon n\equiv b\mod W}\}$, and set $g_{P_0}=g$ and $F_{P_0}=F$. Then~\eqref{bias-induction} follows immediately from~\eqref{4.1}, while~\eqref{4.3} and~\eqref{ass-lip-varying} are trivial.

Assume now that the claim holds for $j-1$, and we prove it for $j$. By~\eqref{bias-induction} and the pigeonhole principle, there are $\gg \exp\bigbrac{ -12(\log(1/\eta))^{C_k^{10^{j-1}}}}$ proportion of progressions $P^{(j-1)} \in \mathcal P^{(j-1)}$ such that
\[
\Bigabs{ \E_{n\in P^{(j-1)}} f(n) F_{P^{(j-1)}} (g_{P^{(j-1)}} (\tfrac{n-b}{W}) \Gamma)} \gg \exp\bigbrac{ -12(\log(1/\eta))^{C_k^{10^{j-1}}}}.
\]

Fix such a progression $P^{(j-1)}$. Since $f$ is a linear combination of at most $(\log N)^C$ Type~I and Type~II sums, and since $\exp\exp((\log \log N)^{1/C_k}) \leq 1/\eta$, another pigeonhole argument yields a single Type~I or Type~II function $h$ such that
\[
\Bigabs{ \E_{n\in P^{(j-1)}} h(n) F_{P^{(j-1)}} (g_{P^{(j-1)}} (\tfrac{n-b}{W}) \Gamma)} \gg \exp\bigbrac{ -20(\log(1/\eta))^{C_k^{10^{j-1}}}}.
\]

Applying the Type~I/II estimates in Lemmas~\ref{type-i-ii}--\ref{type-ii}, we obtain a factorisation
\[
g_{P^{(j-1)}}=\varepsilon'\, g'\,\gamma',
\]
where $g'$ takes values in an $\exp\bigbrac{(\log(1/\eta))^{C_k^{10^j}}}$-rational subgroup of $G$ of strictly smaller step than the subgroup underlying $g_{P^{(j-1)}}$ (in particular, at most $k-1-j$), $\varepsilon'$ is $(\exp\bigbrac{ (\log(1/\eta))^{C_k^{10^{j}}}}, N/W)$-smooth, and $\gamma'$ is $\exp\bigbrac{ (\log(1/\eta))^{C_k^{10^{j}}}}$-rational. Let $q'$ be the period of $\gamma'$, then $q'\leq \exp\bigbrac{ (\log(1/\eta))^{C_k^{10^{j}}}}$. Let $q$ be the product of $q'$ and the common difference of the progressions in $\mathcal P^{(j-1)}$; by the induction hypothesis this common difference is at most $W\exp\bigbrac{ 5(\log(1/\eta))^{C_k^{10^{j-1}}}}$. Hence, without loss of generality, we can assume that
\[
q\leq W\exp\bigbrac{ 2(\log(1/\eta))^{C_k^{10^{j}}}}.
\]

We now subdivide each $P^{(j-1)}\in\mathcal P^{(j-1)}$ into subprogressions of common difference $q$ and length at least $(N/W) \exp\bigbrac{ -5(\log(1/\eta))^{C_k^{10^{j}}}}$. This produces a refinement of $\mathcal P^{(j-1)}$, namely $\mathcal P^{(j)}$. Hence we have
\begin{align}\label{split-pro}
	&\sum_{P^{(j-1)} \in \mathcal P^{(j-1)}}  \Bigabs{ \sum_{n\in P^{(j-1)}} f(n) F_{P^{(j-1)}} (g_{P^{(j-1)}} (\tfrac{n-b}{W}) \Gamma)} \nonumber\\
	=& \sum_{P^{(j-1)} \in \mathcal P^{(j-1)}}\Bigabs{\sum_{P^{(j)} \subset P^{(j-1)}} \sum_{n\in P^{(j)}} f(n) F_{P^{(j-1)}} (g_{P^{(j-1)}} (\tfrac{n-b}{W}) \Gamma)} \nonumber\\
	\leq& \sum_{P^{(j)} \in \mathcal P^{(j)}} \bigabs{ \sum_{n\in P^{(j)} } f(n) F_{P^{(j-1)} } (g_{P^{(j-1)}} (\tfrac{n-b}{W}) \Gamma)}.
\end{align}

For each progression $P^{(j)} \in \mathcal P^{(j)}$, we now fix a representative point $n_{P^{(j)}} \in P^{(j)}$. Then for all $n\in P^{(j)}$ we have
\[
\bigabs{n-n_{P^{(j)}}} \leq (N/W) \exp\bigbrac{ -3 (\log(1/\eta))^{C_k^{10^j}}}.
\]

Thus, it follows from Definition~\ref{rational-smooth} that
\[
\gamma'\!\left(\tfrac{n-b}{W}\right)\Gamma = \gamma'\!\left(\tfrac{n_{P^{(j)}}-b}{W}\right)\Gamma,
\]
and
\[
d\bigbrac{\varepsilon'\!\left(\tfrac{n-b}{W}\right), \varepsilon'\!\left(\tfrac{n_{P^{(j)}}-b}{W}\right)} \leq
\frac{(|n-n_{P^{(j)}}|/W)\cdot \exp \bigbrac{ (\log(1/\eta))^{C_k^{10^j}}}}{N/W}
\leq \exp \bigbrac{ -2(\log(1/\eta))^{C_k^{10^j}}}.
\]

Set $\varepsilon'_{P^{(j)}}\coloneqq \varepsilon'(\tfrac{n_{P^{(j)}}-b}{W})$ and $\gamma'_{P^{(j)}}\coloneqq \gamma'(\tfrac{n_{P^{(j)}}-b}{W})$.
Using the Lipschitz bound of $F_{P^{(j-1)}}$, we deduce that for all $n\in P^{(j)}$ the following inequality holds:
\begin{align*}
& \bigabs{  F_{P^{(j-1)} } \!\left(\varepsilon'\!\left(\tfrac{n-b}{W}\right) g'\!\left(\tfrac{n-b}{W}\right)\gamma'\!\left(\tfrac{n-b}{W}\right) \Gamma\right) -  F_{P^{(j-1)} } \!\left( \varepsilon'_{P^{(j)}} g'\!\left(\tfrac{n-b}{W}\right) \gamma'_{P^{(j)}} \Gamma\right)}\\
=& \bigabs{ F_{P^{(j-1)} } \!\left(\varepsilon'\!\left(\tfrac{n-b}{W}\right) g'\!\left(\tfrac{n-b}{W}\right)\gamma'_{P^{(j)}} \Gamma\right) -  F_{P^{(j-1)} } \!\left( \varepsilon'_{P^{(j)}} g'\!\left(\tfrac{n-b}{W}\right) \gamma'_{P^{(j)}} \Gamma\right) }\\
\leq & d \bigbrac{ \varepsilon'\!\left(\tfrac{n-b}{W}\right) g'\!\left(\tfrac{n-b}{W}\right)\gamma'_{P^{(j)}}, \varepsilon'_{P^{(j)}} g'\!\left(\tfrac{n-b}{W}\right) \gamma'_{P^{(j)}}}\\
 = &d\bigbrac{\varepsilon'\!\left(\tfrac{n-b}{W}\right), \varepsilon'_{P^{(j)}}} \leq \exp \bigbrac{ -2(\log(1/\eta))^{C_k^{10^j}}}.
\end{align*}

Thus, combining assumption (2), inequality~\eqref{bias-induction} at level $j-1$,~\eqref{split-pro}, and the above inequality, we have
\[
2\exp \bigbrac{ -(\log(1/\eta))^{C_k^{10^j}}}(N/W) \leq \sum_{P^{(j)}\in \mathcal P^{(j)}} \bigabs{ \sum_{n\in P^{(j)}} f(n)F_{P^{(j-1)} } \!\left( \varepsilon'_{P^{(j)}} g'\!\left(\tfrac{n-b}{W}\right) \gamma'_{P^{(j)}} \Gamma\right)}.
\]

We now factorise $\gamma'_{P^{(j)}} =\set{\gamma'_{P^{(j)}}} [\gamma'_{P^{(j)}}]$ where $[\gamma'_{P^{(j)}}] \in \Gamma$ and $|\psi(\set{\gamma'_{P^{(j)}}})| \leq 1/2$ for the Mal'cev coordinate map $\psi$. By writing
\[
g_{P^{(j)}}(n) =\set{ \gamma'_{P^{(j)}}}^{-1}g'(n)\set{\gamma'_{P^{(j)}}}
\quad\textnormal{ and }\quad
F_{P^{(j)}} = F_{P^{(j-1)}} (\varepsilon'_{P^{(j)}} \set{\gamma'_{P^{(j)}}} \cdot),
\]
we have from the preceding inequality that
\begin{align}\label{lip-slow}
	\bigabs{ F_{P^{(j)}} (g_{P^{(j)}}(\tfrac{n-b}{W})\Gamma) - F_{P^{(j-1)}}(g_{P^{(j-1)}}(\tfrac{n-b}{W}) \Gamma)  } \leq \exp \bigbrac{ -2(\log(1/\eta))^{C_k^{10^j}}}
\end{align}
provided that $n \in P^{(j)}$.

Therefore, combining the previous two inequalities, we obtain
\[
 \exp \bigbrac{ -(\log(1/\eta))^{C_k^{10^j}}} (N/W) \leq \sum_{P^{(j)} \in \mathcal P^{(j)}} \Bigabs{\sum_{n\in P^{(j)}} f(n) F_{P^{(j)}} (g_{P^{(j)}} (\tfrac{n-b}{W}) \Gamma)}.
\]

Using assumption (2) and the assumption $(\log(1/\eta))^{200C_k^{10^k}} \leq \log N$, the pigeonhole principle and the fact that
\[
|P^{(j)}| \geq (N/W) \exp \bigbrac{ -5(\log(1/\eta))^{C_k^{10^j}}}\geq (N/W)\exp(-(\log N)^{1/100}),
\]
we have from the above inequality that there are $\gg \exp \bigbrac{ -6(\log(1/\eta))^{C_k^{10^j}}}$ proportion of progressions $P^{(j)} \in \mathcal P^{(j)}$ such that
\[
\Bigabs{ \E_{n\in P^{(j)}} f(n) F_{P^{(j)}} (g_{P^{(j)}} (\tfrac{n-b}{W}) \Gamma) }\geq \exp \bigbrac{ -6(\log(1/\eta))^{C_k^{10^j}}}.
\]

In particular,
\[
\E_{P^{(j)} \in \mathcal P^{(j)}} \bigabs{\E_{n \in P^{(j)}} f(n) F_{P^{(j)}} (g_{P^{(j)}} (\tfrac{n-b}{W}) \Gamma)} \geq \exp\bigbrac{ -12(\log(1/\eta))^{C_k^{10^{j}}}},
\]
which is inequality~\eqref{bias-induction}.

Meanwhile, combining~\eqref{lip-slow} with the induction hypothesis~\eqref{ass-lip-varying} at level $j-1$,
we obtain that for all $n\in P^{(j)}$,
\begin{align*}
\Bigabs{F(g(\tfrac{n-b}{W})\Gamma) -F_{P^{(j)}} (g_{P^{(j)}}(\tfrac{n-b}{W})\Gamma)}&\leq \exp\bigbrac{-\frac{1}{j} (\log(1/\eta))^{C_k^{10}}} + \exp\bigbrac{ -2(\log(1/\eta))^{C_k^{10^{j}}}} \\
&\leq \exp\bigbrac{-\tfrac{1}{j+1} (\log(1/\eta))^{C_k^{10}}}.
\end{align*}
This establishes~\eqref{ass-lip-varying} for level $j$.

It remains to verify~\eqref{4.3}. By construction, $\mathcal P^{(j)}$ is a refinement of $\mathcal P^{(j-1)}$, so splitting each $P^{(j-1)}$ into its subprogressions $P^{(j)}\subseteq P^{(j-1)}$ and using~\eqref{lip-slow} together with assumption~(2), we obtain
\begin{align*}
\sum_{P^{(j-1)} \in \mathcal P^{(j-1)}}& \Bigabs{\sum_{n\in P^{(j-1)}} f(n) F_{P^{(j-1)}} (g_{P^{(j-1)}}(\tfrac{n-b}{W}) \Gamma)}\\
\leq& \sum_{P^{(j)} \in \mathcal P^{(j)}} \Bigabs{\sum_{n\in P^{(j)}} f(n) F_{P^{(j)}} (g_{P^{(j)}}(\tfrac{n-b}{W}) \Gamma)} +C(N/W)\exp\bigbrac{-2(\log(1/\eta))^{C_k^{10^j}}}.
\end{align*}

Substituting this into inequality~\eqref{4.3} at level $j-1$ yields~\eqref{4.3} at level $j$. This completes the induction and hence the proof of the claim.

Now take $j=k-1$ and set $\mathcal P\coloneqq \mathcal P^{(k-1)}$. Then the polynomial sequence $g_{P}(\tfrac{n-b}{W})$ takes values in a subgroup of step $0$, hence $g_{P}(\tfrac{n-b}{W})\Gamma$ is constant on $P$. In particular, for each $P\in\mathcal P$ the quantity
\[
C_P\coloneqq F_P\bigl(g_P(\tfrac{n-b}{W})\Gamma\bigr)
\qquad (n\in P)
\]
is well-defined.  Thus, we deduce from~\eqref{4.3} and assumption (2) that
\begin{align*}
\Bigabs{\sum_{n\in [N]\atop n\equiv b\mod W} f(n) F(g(\tfrac{n-b}{W})\Gamma)} & \leq \sum_{P\in \mathcal P}|C_P| \bigabs{\sum_{n\in P} f(n)}	+ O\bigbrac{(N/W)\exp(-\frac{1}{k}(\log(1/\eta))^{C_k^{10}})}\\
& \leq C\sum_{P\in \mathcal P}|\sum_{n\in P}C_P| +O\bigbrac{(N/W)\exp(-2(\log(1/\eta))^{C_k})}\\
& \leq C\sum_{P\in \mathcal P}\bigabs{\sum_{n\in P} F_P(g_P(\tfrac{n-b}{W})\Gamma)} + O\bigbrac{(N/W)\exp(-2(\log(1/\eta))^{C_k})}. 
\end{align*}

On the other hand, by~\eqref{ass-lip-varying} (with $j=k-1$) we have
\begin{align*}
\sum_{P\in \mathcal P} \Bigabs{ \sum_{n\in P} \bigbrac{F(g(\tfrac{n-b}{W})\Gamma) - F_P(g_P(\tfrac{n-b}{W})\Gamma)} }&\leq \sum_{n\in [N]\atop n\equiv b \mod W}\bigabs{ F(g(\tfrac{n-b}{W})\Gamma) -F_P(g_P(\tfrac{n-b}{W})\Gamma)}\\
&\leq  (N/W)\exp(-2(\log(1/\eta))^{C_k}).
\end{align*}

 Combining the above two inequalities, we thus have
 \[
 \bigabs{\sum_{n\in [N]\atop n\equiv b\mod W} f(n) F(g(\tfrac{n-b}{W})\Gamma)} \leq C \sum_{P\in \mathcal P} \bigabs{\sum_{n\in P} F(g(\tfrac{n-b}{W})\Gamma)} +O\Bigbrac{(N/W)\exp(-(\log(1/\eta))^{C_k})}.
 \]

This completes the proof in the correlation case. Together with the non-correlation alternative, the proposition follows.
\end{proof}

We wish to apply Proposition~\ref{cor-1} to a $W$-tricked version of the von Mangoldt function. Relating to assumption (2) of Proposition~\ref{cor-1}, we have the following bound.

\begin{lemma}[Brun--Titchmarsh bound]\label{bt}
Let $N\geq 2$ and let $w=(\log N)^{c_0}$ for some small constant $c_0<1/100$, and set $W=\prod_{p\leq w} p$. Let $1\leq b\leq W$ with $(b,W)=1$.
Let  $P\subseteq[\frac{N-b}{W}]$ be an arithmetic progression of  length $\geq (N/W)\exp(-2w)$. Then we have
\[
\frac{\phi(W)}{W} \E_{n\in P} \Lambda (Wn+b) \ll  1.
\]	
\end{lemma}

\begin{proof}
Without loss of generality, we may assume that $P=x_0+q\cdot [X]\subseteq [\frac{N-b}{W}]$ with $q\leq \exp(2w)$ and $X\geq (N/W) \exp(-2w)$. It follows from the Brun--Titchmarsh theorem (cf. \cite[Theorem 6.6]{IK}) that
\[
\frac{\phi(W)}{W}\sum_{n\in P} \Lambda (Wn+b) \leq \frac{\phi(W)}{W} \sup_{a\mod{Wq}}\sum_{\substack{Wx_0+b < n\leq W x_0+ WqX+b \\ n\equiv a \mod{Wq}}} \Lambda (n) \ll \frac{\phi(W)}{W}\frac{WqX}{\phi(Wq)}.
\]

Dividing by $|P|=X$ gives
\[
\frac{\phi(W)}{W}\E_{n\in P} \Lambda (Wn+b) \ll \frac{\phi(W)}{W}\frac{Wq}{\phi(Wq)} =\prod_{\substack{p\mid q\\ p>w}}\left(1-\frac{1}{p}\right)^{-1}.
\]

 It follows from Taylor expansion that
\[
\prod_{\substack{p\mid q \\ p>w}} \left(1-\frac{1}{p}\right)^{-1} =\exp \Biggbrac{ \sum_{\substack{p\mid q \\ p>w}}(p^{-1} +O(p^{-2})) } \ll \exp\left(\frac{\omega (q)}{w}\right),
\]
where $\omega (q)$ denotes the number of prime factors of $q$, which is bounded by $O(\frac{\log q}{\log \log q})$. We thus conclude that this is $\ll 1$ in light of the assumption $\log q\leq 2w$. The claim follows.
\end{proof}

We are now ready to verify the correlation condition~\eqref{corr-nil-f}.

\begin{lemma} \label{cor-mangoldt}

Let $w=(\log N)^{c_0}$ for some small constant $c_0>0$ and set $W=\prod_{p\leq w}p$. Let $1\leq b\leq W$ with $(b,W)=1$.
Let $F(g(\cdot)\Gamma)$ be a nilsequence  of degree at most $k-1$, complexity  $M \leq\exp((\log N)^{c_k})$  and dimension  $d\leq (\log N)^{c_k}$, where $c_k>0$ is a constant sufficiently small compared to $c_0$ (in terms of $k$). Then  there exists a partition $\mathcal P$ of $[\frac{N-b}{W}]$ into arithmetic progressions, each of length at least $N \exp(-2w)$, and such that
\begin{align}\label{eq:corrbound}
	\sum_{n\in [\frac{N-b}{W}]}\frac{\phi(W)}{W}\Lambda(Wn+b) F(g(n)\Gamma)
	\ll  \sum_{P\in \mathcal P}\bigabs{\sum_{n\in P} F(g(n)\Gamma)} + N\exp(-(\log  N)^{c_k})/W.
\end{align}

\end{lemma}

\begin{proof}
First,  making the change of variables $Wn+b\to n$, we have
\[
\sum_{n\in [\frac{N-b}{W}]}\frac{\phi(W)}{W}\Lambda(Wn+b) F(g(n)\Gamma)=\sum_{n\in [N]} \frac{\phi(W)}{W} \Lambda(n)1_{n \equiv b \mod W} F(g(\frac{n-b}{W})\Gamma).
\]

By Vaughan's identity, for $n\in \mathbb{N}$ we have
	\[
	\Lambda (n)= 	\Lambda (n)1_{n\leq N^{1/3}}-\sum_{d\leq N^{2/3}} a_d1_{d\mid n}+ \sum_{d\leq N^{1/3}} \mu(d) 1_{d\mid n} \log \frac{n}{d}+\sum_{d,m>N^{1/3}} \Lambda(d) b_m 1_{n=dm},
	\]
	where $a_d=\mu 1_{[1,N^{1/3}]}*\Lambda 1_{[1,N^{1/3}]}$ and $b_m=\mu 1_{(N^{1/3},\infty)}*1$ are divisor-bounded functions. 
    
    Using the identities
    \begin{align*}
      1_I(d)=\sum_{\ell\geq 0}\int_{1}^{N}\frac{1_{[2^{\ell},2^{\ell+1})\cap I}(d)}{\log N}\frac{\rd t}{t} \quad \textnormal{ and } \quad \log \frac{n}{d} = \int_1^N \left(1-1_{t>n}-\frac{\log d}{\log N}\right) \frac{\rd t}{t}
    \end{align*}
    valid for any interval $I$
    and the fact that $\Lambda1_{[1,N^{1/3}]}(n)=\sum_{d\mid n}1_{\{1\}}(d)\Lambda1_{[1,N^{1/3}]}(n/d)$, we conclude that for $n\in [N]$ we have
    \begin{align*}
    \Lambda(n)=\sum_{0\leq \ell\leq (\log N)/(\log 2)}\int_{1}^{N}(f_{\ell,t}(n)+g_{\ell,t}(n))\frac{\rd t}{t},    
    \end{align*}
    where for each $\ell$ and $t$, the function $f_{\ell,t}$ is a Type I sum of level $N^{2/3}$ at scale $N$ and the function $g_{\ell,t}$ is a Type II sum of range $[N^{1/3}/2,N^{2/3}]$ at scale $N$.

    Now set $1/\eta=\exp((\log N)^{c_k})$ in Proposition~\ref{cor-1}. Since $c_k$ is sufficiently small (relative to $c_0$), we can ensure that $\exp\bigbrac{5(\log 1/\eta)^{C_k^{10^k}}}\leq W^{1/2}$, say. Therefore, if $P\subset [N]$ is an arithmetic progression of length 
    \[
    |P|\geq (N/W)\exp\bigbrac{-5(\log 1/\eta)^{C_k^{10^k}}}\geq N\exp(-2w)
    \] 
    and with common difference $q$ divisible by $W$, then Lemma~\ref{bt} gives
    \[
    \sum_{n\in P} \frac{\phi(W)}{W}\Lambda (n) \ll |P|.
    \]

    Hence the hypotheses of Proposition~\ref{cor-1} are satisfied. It follows that there exists a partition $\mathcal P$ of $\set{n\in [N]\colon  n\equiv b\mod W}$ into progressions, each of cardinality at least $N\exp(-2w)$, such that
    \[
    \sum_{n\in [N]} \frac{\phi(W)}{W} \Lambda(n)1_{n \equiv b \mod W} F(g(\frac{n-b}{W})\Gamma)\ll  \sum_{P\in \mathcal P}\bigabs{\sum_{n\in P} F(g(\frac{n-b}{W})\Gamma)} + \frac{N\exp(-(\log  N)^{c_k})}{W}.
    \]

   Translating back to the original variable $m=\frac{n-b}{W}$ yields~\eqref{eq:corrbound}.

\end{proof}

\section{Generalised von Neumann theorem}\label{sec5}

To use the quantitative transferred weak regularity lemma (Proposition~\ref{tra-wea-2}), we need a quantitative version of the generalised von Neumann theorem for unbounded functions. A qualitative version was established by Green and Tao in \cite[Proposition 7.1]{GT-linear} (and in~\cite[Proposition 5.3]{GT-prime}). In \cite[Theorem 9.2]{TT}, it was sketched how this can be improved to a double-logarithmic error term. The proof there would work for logarithmic savings as well; here we have chosen to give more details for the arguments to make the paper more self-contained, as well as to track precisely the order of the Gowers norm required to control the relevant averages for potential future applications.

We need the following quantitative linear forms condition.

\begin{definition}[Linear forms condition]
Let $N\geq 1$. Let $m,d,L$ be positive integer parameters. Let $0<\eps<1$ be sufficiently small depending on $m,d,L$. Let $\nu\colon \mathbb{Z} \to\R_{\geq 0}$ be a non-negative function. We say that $\nu$ satisfies the $(m,d,L)$-\emph{linear forms condition} at scale $N$ with error $\eps$ if the following holds. For any $1\leq d'\leq d$, $1\leq m'\leq m$ and any finite complexity system\footnote{Recall that a system $\Psi$ is of finite complexity if for each $i$, the linear part of $\psi_i$ does not lie in the $\mathbb{Q}$-linear span of the linear parts of $\{\psi_j\colon j\neq i\}$.} $\Psi=(\psi_1,\dots,\psi_{m'})$ of affine-linear forms on $\Z^{d'}$ with all coefficients (including constant terms) of the forms $\psi_i$ bounded in magnitude by $L$, we have
\[
\left|\E_{\vec n \in [N]^{d'}} \prod_{i\in [m']} \nu (\psi_i(\vec n)) -1\right|\leq \eps.
\]
\end{definition}

\begin{lemma}[A quantitative generalised von Neumann theorem]\label{neumann}
Let $k\geq 3$ be a natural number and $0<\eps<1$. Then the following statement holds for some large constant $D_k>1$. Let $\nu\colon \mathbb Z\to\mathbb R_{\ge 0}$ be a non-negative function satisfying the $(D_k,D_k,D_k)$-linear forms condition at scale $N$ with error $\eps$. Let $c_1,\dots,c_k$ be distinct integers with $|c_i|\le D_k/10$, and set
\[
c\coloneqq \max_{1\leq i,j\leq k}|c_i-c_j|.
\]
Suppose that $f_1,\dots,f_k,g_1,\dots,g_k\colon \mathbb Z\to\mathbb C$ are $\nu$-bounded functions and supported on $[N]$. Then there exists a constant $\gamma_k>0$ such that
\begin{align*}
	N^{-1} \sum_{x\in \Z}\E_{d\in [N]} \prod_{1\leq i\leq k} f_i(x+c_id) =&\ N^{-1} \sum_{x\in \Z}\E_{d\in [N]} \prod_{1\leq i\leq k} g_i(x+c_id)\\
	 &+O_{k,c_1,\dots,c_k}\bigbrac{\eps^{\gamma_k}+\max_{i\in [k]}\norm{f_i-g_i}_{U^{k-1}[cN]}}.
\end{align*}
\end{lemma}

\begin{proof}
One may expand the difference between the two multilinear averages by a standard telescoping identity. Namely,
\begin{align*}
N^{-1} \sum_{x\in \Z}&\E_{d\in [N]}\prod_{1\leq i\leq k} f_i(x+c_i d)  = N^{-1} \sum_{x\in \Z}\E_{d\in [N]} \prod_{1\leq i\leq k} g_i(x+c_id)\\
+& \sum_{1\leq j\leq k} N^{-1} \sum_{x\in \Z}\E_{d \in [N]} 	 \prod_{i<j} f_i(x+c_id ) (f_j-g_j) (x+c_jd) \prod_{j<i\leq k} g_i(x+c_id).
\end{align*}

Taking absolute values, it suffices to show that for every $1\le j\le k$ and every choice of $\nu$-bounded functions $l_1,\dots,l_k\colon \mathbb Z\to\mathbb C$ supported on $[N]$, one has
\[
\Bigabs{N^{-1} \sum_{x\in \Z}\E_{d\in [N]} \prod_{1\leq i\leq k}l_i(x+c_id)}\ll_{k,c_1,\dots,c_k} \eps^{\gamma_k} + \norm{l_j}_{U^{k-1}[cN]}.
\]

By relabelling indices, we may assume without loss of generality that $j=k$. Thus it is enough to prove
\begin{align}\label{group-setting}
\Bigabs{N^{-1}\sum_{x\in \Z}\E_{d\in [N]}  \prod_{1\leq i\leq k}l_i(x+c_id)}\ll_{k,c_1,\dots,c_k}  \eps^{\gamma_k} + \norm{l_k}_{U^{k-1}[cN]}.
\end{align}

For the Cauchy--Schwarz iteration, we introduce the following notation. For an integer $j>0$ and a vector $\vec h=(h_1,\dots,h_j)$, define $\Delta (l_i;\vec h)=\Delta_{(c_i-c_1)h_1,\dots,(c_i-c_{j})h_j} l_i$, where the difference operator was defined in~\eqref{eq:deltadef}. Besides, we also adopt the conventions $\Delta(l_i;\emptyset)=l_i$ and $c_0=0$. Given any $0\leq j\leq k-2$ we are going to prove that
\begin{align}\label{ind-nuem}
	\begin{split}
	&\quad \Bigabs{ \frac{1}{N}\sum_{x\in \mathbb{Z}} \E_{d\in [N]} \E_{\vec h\in [-N,N]^j}\prod_{j+1\leq m\leq k} \Delta(l_m;\vec h) (x+(c_m-c_j)d)}^2
	\ll_{k} \eps +\\
	&+ \frac{1}{N}\sum_{x\in \mathbb{Z}} \E_{d\in [N]} \E_{\vec h\in [-N,N]^j}\E_{|h_{j+1}| \leq N}
	 \prod_{j+2\leq m\leq k}\Delta_{(c_m-c_{j+1})h_{j+1}} \Delta(l_m;\vec h) (x+(c_{m}-c_{j+1})d)  .
	\end{split}
\end{align}

First, note that when $j=0$, the expression becomes
\[
  \frac{1}{N}\sum_{x\in \mathbb{Z}} \E_{d\in [N]} \prod_{1\leq m \leq k} \Delta(l_m;\vec 0)(x+c_md) = \frac{1}{N}\sum_{x\in \mathbb{Z}} \E_{d\in [N]} \prod_{1\leq m \leq k}l_m(x+c_md),
\]
which is exactly the left-hand side of~\eqref{group-setting}.

We iterate~\eqref{ind-nuem} for $j=0,1,\dots,k-2$. After $k-1$ steps, the right-hand side reduces to an average involving only $l_k$. More precisely, applying~\eqref{ind-nuem} a total of $k-1$ times (from $0$ to $k-2$), with the observation that $j=0$ yields the initial expression, we thus have
\begin{align*}
&\Bigabs{  N^{-1}\sum_{x\in \Z}\E_{d\in [N]}  \prod_{1\leq i\leq k}l_i(x+c_id)}^{2^{k-1}} \\
	&\ll_{k} \eps+ N^{-1}\sum_{x\in \Z}\E_{d\in [N]}\E_{ h_1,\dots,h_{k-1} \in [-N,N] } \Delta_{(c_k-c_1)h_1,\dots,(c_k-c_{k-1})h_{k-1} } l_{k} (x+(c_k-c_{k-1})d).
\end{align*}

It is harmless to shift the variable $x+(c_k-c_{k-1})d \to x$ in the second term, so it suffices to prove that
\[
N^{-1}\sum_{x\in \Z}\E_{ h_1,\dots,h_{k-1} \in [-N,N] } \Delta_{(c_k-c_1)h_1,\dots,(c_k-c_{k-1})h_{k-1} } l_{k} (x)\ll_{k,c_1,\dots,c_k} \norm{l_k}_{U^{k-1}[cN]}^{2^{k-1}}.
\]

First, make the change of variables $h_i'=(c_k-c_i)h_i$ for $1\leq i\le k-1$, and remove the congruence conditions using orthogonality, yielding
\begin{align*}
\ll&_{k,c_1,\dots,c_k} N^{-1}\sum_{x\in \Z} \E_{|h_i|\leq |c_k-c_i|N \atop 1\leq i\leq k-1} \prod_{1\leq i\leq k-1} 1_{h_i\equiv 0\mod{|c_k-c_i|}}\Delta_{h_1,\dots,h_{k-1}} l_k(x)\\
\ll&_{k,c_1,\dots,c_k} N^{-1}\sum_{x\in \Z} \E_{|h_i|\leq |c_k-c_i|N \atop 1\leq i\leq k-1}  \prod_{1\leq i\leq k-1} \left( \E_{1\leq r_i\leq |c_k-c_i|} e\left( \frac{h_ir_i}{|c_k-c_i|}  \right)\right)\Delta_{h_1,\dots,h_{k-1}} l_k(x)\\
\ll&_{k,c_1,\dots,c_k} \E_{1\leq r_i\leq |c_k-c_i| \atop 1\leq i\leq k-1} N^{-1} \sum_{x\in \Z} \E_{|h_i|\leq |c_k-c_i|N \atop 1\leq i\leq k-1} \Delta_{h_1,\dots,h_{k-1}} l_k(x) e\left(\sum_{1\leq j\leq k-1}\frac{r_jh_j}{|c_k-c_j|} \right) .
\end{align*}

Since $\supp(l_k)\subseteq [N]$, for arbitrary $r_1,\dots,r_{k-1}\in \Z$ we re-arrange the inner sum to see that
\begin{align*}
&N^{-1} \sum_{x\in \Z} \E_{|h_i|\leq |c_k-c_i|N \atop 1\leq i\leq k-1} \Delta_{h_1,\dots,h_{k-1}} l_k(x) e\left(\sum_{1\leq j\leq k-1}\frac{r_jh_j}{|c_k-c_j|} \right)\\
=& N^{-1} \sum_{x\in \Z} \E_{|h_i|\leq |c_k-c_i|N \atop 1\leq i\leq k-1} \prod_{\omega \in \set{0,1}^{k-1} \atop |\omega|\neq 1}\mathcal{C}^{|\omega|}l_k(x+\omega \cdot\vec h) \prod_{1\leq j\leq k-1} \overline{l_k(x+h_j)} e\left(\frac{r_jh_j}{|c_k-c_j|}\right)\\
\ll&_{k,c_1,\dots,c_k} \E_{x\in [cN]} \E_{|h_i|\leq cN \atop 1\leq i\leq k-1}l_k(x)e\left(-\sum_{1\leq j\leq k-1}\frac{r_jx}{|c_k-c_j|}\right) \prod_{\omega \in \set{0,1}^{k-1} \atop |\omega|\neq 0, 1}\mathcal{C}^{|\omega|}l_k(x+\omega \cdot\vec h) \\
&\qquad \qquad \qquad \qquad \qquad \qquad\prod_{1\leq j\leq k-1} \overline{l_k(x+h_j)} e\left(\frac{r_j(x+h_j)}{|c_k-c_j|}\right)\\
\ll&_{k,c_1,\dots,c_k} \norm{l_k}_{U^{k-1}[cN]}^{2^{k-1}},
\end{align*}
where the last inequality follows from an application of the Gowers--Cauchy--Schwarz inequality together with the fact that multiplication by a linear phase does not affect the $U^{k-1}$-norm. Averaging over $(r_1,\dots,r_{k-1})$ and substituting back into the previous estimate, we obtain
\[
\frac{1}{N}\sum_{x\in\mathbb Z}\E_{|h_i|\le |c_k-c_i|N}\Biggl(\prod_{i=1}^{k-1}1_{|c_k-c_i|\mid h_i}\Biggr)\,\Delta_{h_1,\dots,h_{k-1}}l_k(x)\ll_{k,c_1,\dots,c_k} \|l_k\|_{U^{k-1}[cN]}^{2^{k-1}}.
\]
This establishes~\eqref{group-setting}.

We now prove inequality~\eqref{ind-nuem}. Making the change of variables $x+(c_{j+1}-c_j)d \to x$, we see that the left-hand side of~\eqref{ind-nuem} is equal to
\[
\Bigabs{N^{-1} \sum_{x\in \Z} \E_{\vec h\in [-N,N]^j} \E_{d\in [N]} \Delta (l_{j+1};\vec h)(x) \prod_{j+2\leq m\leq k} \Delta (l_m;\vec h)(x+(c_m-c_{j+1})d)}^2.
\]

Applying the Cauchy--Schwarz inequality, together with the assumption that $|l_{j+1}| \leq \nu$ pointwise, we find that the above expression is bounded by
\begin{multline}\label{square-out}
\E_{\vec h\in [-N,N]^j} N^{-1} \sum_{x\in \Z} \Delta_{(c_{j+1} -c_1)h_1,\dots,(c_{j+1}-c_j)h_j} \nu(x) \times \E_{\vec h\in [-N,N]^j} N^{-1}\sum_{x\in \Z} \\
	\Delta_{(c_{j+1} -c_1)h_1,\dots,(c_{j+1}-c_j)h_j} \nu(x) \Bigabs{ \E_{d\in [N]} \prod_{m=j+2}^k \Delta(l_m;\vec h)(x+(c_m-c_{j+1})d)}^2.
\end{multline}

Next, we use the $(D_k,D_k,D_k)$-linear forms condition.
Since the coefficients $c_1,\dots,c_k$ are distinct, the linear forms appearing in the expansion of the difference $\Delta_{(c_{j+1}-c_1)h_1,\dots,(c_{j+1}-c_j)h_j}\nu(x)$ are pairwise distinct and have bounded complexity (with constants depending only on $k$ and the $c_i$). Therefore the linear forms condition implies
\[
N^{-1} \sum_{x\in \Z} \E_{\vec h \in [-N,N]^j} \Delta_{(c_{j+1} -c_1)h_1,\dots,(c_{j+1} -c_j) h_j} \nu(x)= 1+O_k(\eps),
\]
while the linear forms assumption along with an application of the Gowers--Cauchy--Schwarz inequality also leads to
\begin{multline*}
\E_{\vec h\in [-N,N]^j}N^{-1} \sum_{x\in \Z} \bigbrac{\Delta_{(c_{j+1} -c_1)h_1,\dots,(c_{j+1} -c_j) h_j} \nu(x)-1}\times\\
\times\Bigabs{\E_{d\in [N]} \prod_{m=j+2}^k \Delta(l_{m};\vec h) (x+(c_{m} -c_{j+1})d)}^2 \ll_k \eps.
\end{multline*}
Substituting these two estimates into~\eqref{square-out} and expanding
\[
\Delta_{(c_{j+1} -c_1)h_1,\dots,(c_{j+1} -c_j) h_j} \nu(x) =(\Delta_{(c_{j+1} -c_1)h_1,\dots,(c_{j+1} -c_j) h_j} \nu(x)-1)+1,
\]
we conclude that the left-hand side of~\eqref{ind-nuem} is bounded by
\begin{align*}
\ll_{k}&\eps + \E_{\vec h\in [-N,N]^j}
N^{-1}\sum_{x\in \Z}  \Bigabs{\E_{d\in [N]} \prod_{m=j+2}^k \Delta(l_m;\vec h) (x+(c_{m} -c_{j+1})d)}^2.
\end{align*}

Finally, we use the standard identity (a van der Corput/Gowers--Cauchy--Schwarz step)
\[
\Bigabs{\E_{d\in [N]} A(d)}^2 \ll \E_{|h_{j+1}|\leq N}\E_{d\in [N]} A(d)\overline{A(d+h_{j+1})},
\]
valid for any bounded function $A$ supported on $[N]$ (with harmless boundary truncations). Applying this with
\[
A(d)\coloneqq \prod_{m=j+2}^k \Delta(l_m;\vec h) (x+(c_{m} -c_{j+1})d)
\]
gives
\begin{align*}
\ll_{k} & \eps+ \frac{1}{N}\sum_{x\in \Z}\E_{d\in [N]} \E_{\vec h\in[-N,N]^{j}}\E_{|h_{j+1}|\leq N} \prod_{m=j+2}^k\Delta_{(c_m-c_{j+1}) h_{j+1}}\Delta(l_m;\vec h) (x+(c_{m} -c_{j+1})d),
\end{align*}
which is exactly the right-hand side of~\eqref{ind-nuem}.
\end{proof}

We conclude this section with the existence of a pseudorandom majorant for the primes with quantitative error terms. 

\begin{lemma}[Pseudorandomness]\label{pseudo}
Let $N\geq 2$ and let $D\geq 2$. Let $0<c_0<\frac{1}{100}$ be a small constant, set $W=\prod_{p\leq (\log N)^{c_0}}p$, $1\leq b\leq W$ and $(b,W)=1$. Then 
there exists a function $\nu_D\colon\mathbb{Z}\to \R_{\geq 0}$ which satisfies the $(D,D,D)$-linear forms condition at scale $N$ with error $O_D((\log N)^{-c'})$ for some small constant $c'>0$ (with $c'<c_0$), and such that
\[
\tfrac{ \phi(W)}{W}\Lambda(W n+b) +1\ll_D \nu_D(n)
\]
whenever $n\in [\frac{N-b}{W}]$.
In particular, for any $k\geq 2$ and $D$ large enough in terms of $k$, we have
\[
\norm{\nu_D-1}_{U^{2k}[\frac{N-b}{W}]}^{2^{2k}} \ll_{k,D} (\log N)^{-c'}.
\]
\end{lemma}

\begin{proof}
Let $\Lambda'$ be the restriction of $\Lambda$ to the primes. Since perfect powers form a very sparse set, it suffices to prove that for $Wn+b\in [N^{1/2},N]\cap \mathbb{N}$ we have
\[
\tfrac{ \phi(W)}{W}\Lambda'(W n+b) +1\ll_{D}\nu_D'(n)
\]
for some function $\nu_D'\colon \mathbb{Z}\to \mathbb{R}_{\geq 0}$ satisfying the $(D,D,D)$-linear forms condition at scale $N$ with error $O_D((\log N)^{-c'})$.
Indeed, setting
\[
\nu_D(n)\coloneqq \nu_D'(n)+(\log N)1_{n\ \textnormal{a perfect power}},
\]
the majorisation $\tfrac{\phi(W)}{W}\Lambda(Wn+b)+1\ll_D \nu_D(n)$ follows, and the additional term does not affect the linear forms condition: for any fixed system of bounded-complexity affine-linear forms $\Psi$, the contribution to
\[
\E_{\vec n\in [N]^{d'}}\prod_{i\in[m']}\nu_D(\psi_i(\vec n))
\]
from tuples for which at least one $\psi_i(\vec n)$ is a perfect power is $O_D((\log N)N^{-1/2})$, since the set of perfect powers up to $O(N)$ has size $O(N^{1/2})$ and the coefficients in $\Psi$ are bounded. This error is negligible compared with $(\log N)^{-c'}$.

Let
\[
\Lambda_{\chi,R,2}(n) =(\log R) \Bigbrac{\sum_{d\mid n} \mu(d) \chi \bigbrac{\tfrac{\log d}{\log R}}}^2,
\]
where $R=N^{\gamma_D}$ for some sufficiently small constant $0<\gamma_D<1/2$, and let $\chi\colon \R \to\R_{\geq 0}$ be any smooth function that is supported on $[-2,2]$ and equals $1$ on $[-1,1]$, and satisfies $\int_{\R} \chi'(x)^2\rd x=1$. Define
\[
\nu_D'(n) =\tfrac{\phi(W)}{W} \Lambda_{\chi,R,2}(Wn+b).
\]
Then $\nu_D'\geq 0$ pointwise.
Since $\gamma_D<1/2$ is small enough, for primes $Wn+b\ge N^{1/2}$ one has $\nu_D'(n)=\frac{\phi(W)}{W}\log R \gg_D \frac{\phi(W)}{W}\log(Wn+b)$, whence
\[
\tfrac{ \phi(W)}{W}\Lambda' (Wn+b) +1\ll_D \nu_D'(n)
\]
for $n\in [\frac{N-b}{W}]$.

It remains to verify the linear forms condition. Fix any $1\leq d,m\leq D$ and any finite complexity system $(\psi_1,\dots,\psi_{m})$ of affine-linear forms. For primes $p\nmid W$, define the local factor
\[
\beta_p =\E_{\vec n\in (\Z/p\Z)^d} \prod_{1\leq i\leq m}\frac{p}{p-1}1_{W\psi_i(\vec n)+b\not\equiv 0\!\!\!\pmod p},
\]
and for primes $p\mid W$ set $\beta_p\coloneqq 1$ (these primes are already absorbed into the $W$-trick, and in the argument below only the Euler factors with $p\nmid W$ play a role).
By the argument in \cite[Proposition 8.4]{TT} (applied with the $W$-tricked weight $\Lambda_{\chi,R,2}(W\cdot+b)$ and the above choice of local factors), one obtains an asymptotic of the form
\[
\E_{\vec n \in [N]^d} \prod_{i\in [m]} \frac{\phi(W)}{W} \Lambda_{\chi,R,2} (W \psi_i(\vec n) +b)
= \prod_{p}\beta_p + O_D((\log N)^{-c'})
\]
for some constant $c'>0$.
Moreover, for $p\nmid W$ the local factors satisfy $\beta_p=1+O_D(p^{-2})$ (using finite complexity and bounded coefficients to control collisions mod $p$), so the Euler product $\prod_p\beta_p$ converges and equals $1+O_D((\log N)^{-c'})$ after truncation at the relevant scales in \cite[Proposition 8.4]{TT}. This yields
\[
\E_{\vec n \in [N]^d} \prod_{i\in [m]} \frac{\phi(W)}{W} \Lambda_{\chi,R,2} (W \psi_i(\vec n) +b)= 1+O_D((\log N)^{-c'}),
\]
as required.
\end{proof}

\section{Transference principle and the sparse Green--Tao theorem}
\label{sec6}

The following lemma is originally due to Varnavides~\cite{Var}; the form stated here is taken from~\cite[Theorem 2.1]{RW}.

\begin{lemma}[Varnavides' bound]\label{1bd-counting}
Let $N\geq 2$ and let $k\geq 4$ be a natural number. Let $0<\delta<1/3$ be a parameter.
Suppose that $g\colon [kN] \to [0,1]$ is a function with average $\E_{n\in [N]} g(n) =\delta $.
Then
\[
\E_{n,d\in [N]} g(n) g(n+d)\cdots g(n+(k-1)d) \gg \exp\bigl(-\exp((\log(1/\delta))^{C_k})\bigr)
\]
for some large constant $C_k>0$. Furthermore, for $k=4$ we have the stronger bound
\[
\E_{n,d\in [N]} g(n) g(n+d)  g(n+2d) g(n+3d) \gg \exp(-\delta^{-C_4}).
\]
\end{lemma}

\begin{proof} We may assume that $N$ is large enough in terms of $k$ by adjusting the constants $C_k$ in the statement.

Let  $A=\{n\in [N]\colon g(n)\geq \delta/2\}$. Then $|A|\geq (\delta/2)N$. Now, since 
\[
\E_{n,d\in [N]} g(n) g(n+d)\cdots g(n+(k-1)d)\geq \left(\frac{\delta}{2}\right)^k \E_{n,d\in [N]} 1_A(n) 1_A(n+d)\cdots 1_A(n+(k-1)d),
\]
from~\cite[Theorem 18 and Remark 1]{shkredov} it follows that we have the claim of the lemma if the right-hand side of the bound is replaced by $(\delta/2)^k\delta^2/N_k(\delta/4)^3$, where $N_k(\rho)$ is the smallest positive integer such that for every integer $M\geq N_k(\rho)$, any subset of $[M]$ of size at least $\rho M$ contains a nontrivial $k$-term arithmetic progression. From the Szemer\'edi bounds given in~\eqref{eq:sz}, for some constants $C_k'>0$ and all $\rho\in (0,1/2)$, we have
\begin{align*}
 N_k(\rho)\ll \begin{cases} \exp(\rho^{-C_4'}),\quad &k=4\\
     \exp(\exp((\log(1/\rho))^{C_k'})),\quad &k\geq 5,
 \end{cases} 
\end{align*}
 and the claim follows on adjusting the constants.
\end{proof}

\begin{theorem}[Counting $k$-APs in subsets of the primes]\label{density-bd-2}
Let $k\geq 4$ be a natural number. Let $N\geq 10$, and let $0<\delta<1/3$ satisfy $\delta\geq \exp(-(\log \log N)^{c'})$ for a sufficiently small constant $c'>0$. Let $0<c_0<1/100$ be a constant, let $W=\prod_{p\leq (\log N)^{c_0}}p$, $1\leq b\leq W$ and $(b,W)=1$. Suppose that  $f\colon \Z\to\R_{\geq 0}$ is supported on the interval $[\frac{N-b}{W}]$ and satisfies $\E_{n\in [\frac{N-b}{W}]}f(n)\geq\delta$, and additionally
\[
0\leq f\leq \frac{\phi(W)}{W}\Lambda(W \cdot +b)
\]
Then there exist constants $c>0$ and $c_k>0$, such that
\begin{multline*}
\frac{W}{N-b}\sum_{n\in \Z} \E_{d \in [\frac{N-b}{W}]} f(n)\cdots f(n+(k-1)d)\\
\geq c \exp\bigl(-\exp((\log(1/\delta))^{C_k})\bigr) -O_k\bigbrac{\exp (-(\log \log N)^{c_k})}.
\end{multline*}
Furthermore, for $k=4$ we have the stronger bound
\begin{align*}
&\frac{W}{N-b}\sum_{n\in \Z} \E_{d \in [\frac{N-b}{W}]} f(n)f(n+d)f(n+2d)f(n+3d)
\geq c \exp(-\delta^{-C_4})\\
&\quad - O\bigbrac{\exp (-(\log \log N)^{c_4})}.
\end{align*}
\end{theorem}

\begin{proof}
First, we restrict the support of $f$ to primes because prime powers form a sparse set; thus, we can assume that
\[
f\leq \frac{\phi(W)}{W}\Lambda'(W \cdot +b)
\]
and $\E_{n\in [\frac{N-b}{W}]} f(n)\geq \delta/2$. Let $D_k$ be large enough in terms of $k$. Lemma~\ref{pseudo} then gives us a function $\nu_1\colon \mathbb{Z}\to\R_{\geq 0}$ such that $f\leq \nu_1$ pointwise and such that $\nu_1$ satisfies the $(D_k,D_k,D_k)$-linear forms condition at scale $(N-b)/W$ with error $O_k((\log N)^{-c_k'})$ for some $c_k'>0$ depending on $D_k$. In particular, we have
\begin{align}\label{pseudo-nu}
\norm{\nu_1 -1}_{U^{2k}[\frac{N-b}{W}]} \ll_k (\log N)^{-\gamma_k}
\end{align}
for some $\gamma_k>0$.

Now set $\eps =(\log N)^{-\gamma_k}$. It follows from~\eqref{pseudo-nu} that the first hypothesis of Proposition~\ref{tra-wea-2} holds with this majorant $\nu_1$. 
Next set $\nu_2=\frac{\phi(W)}{W}\Lambda'(W \cdot +b)1_{[(N-b)/W]}$, then $\nu_2$ is also a majorant of $f$, and moreover $\|\frac{\phi(W)}{W}\Lambda'(W \cdot +b)1_{[(N-b)/W]}\|_\infty\ll \log N$. Thus the second hypothesis of Proposition~\ref{tra-wea-2} holds as well, and it remains only to verify the third. Set
\[
\eta =\exp\bigbrac{-\exp((\log(1/\eps))^{1/2})}.
\]
One readily checks that this choice of $\eta$ satisfies
\[
1/\eta \geq \log N \qquad \text{ and } \qquad \log(1/\eta) \leq \exp((\log\log N)^{1/2})\leq (\log N)^c
\]
for any  small constant $c>0$.
Therefore, Proposition~\ref{cor-1} implies that the third hypothesis of Proposition~\ref{tra-wea-2} also holds. Applying Proposition~\ref{tra-wea-2}, we obtain a small constant $0<c_k<1/5$, a bounded function $g\colon [\frac{N-b}{W}]\to[0,2]$ and a set $\Omega \subset[\frac{N-b}{W}]$. We then extend $g$ to $\mathbb{Z}$ by setting it equal to zero outside $\left[\frac{N-b}{W}\right]$. Then the function $g$ satisfies the following properties:
\begin{enumerate}
\item
\begin{align}\label{uni-bd}
\norm{f-g}_{U^k[\frac{N-b}{W}]} \ll_k \exp(-(\log \log N)^{c_k});
\end{align}
\item and
\begin{align*}
\bigabs{\E_{n\in [\frac{N-b}{W}]} (f+1)1_{[\frac{N-b}{W}]\backslash\Omega}(n)}\leq (\log N)^{-\gamma_k}.
\end{align*}
\end{enumerate}

Combining this with $\E_{n\in [\frac{N-b}{W}]}f(n)\geq\delta/2$ we deduce
\begin{align} \label{density-bd}
\E_{n\in [\frac{N-b}{W}] } g(n) 1_\Omega (n)
&\geq \E_{n\in [\frac{N-b}{W}]} f(n) 1_\Omega (n) -\bigabs{\E_{n\in [\frac{N-b}{W}]} (f-g)(n) 1_\Omega (n)} \nonumber\\
&\geq \E_{n\in [\frac{N-b}{W}] }f(n) - \bigabs{\E_{n\in [\frac{N-b}{W}]} f(n) 1_{[\frac{N-b}{W}]\backslash\Omega} (n)}\\
&\quad - O_k\bigbrac{\exp (-(\log \log N)^{c_k})} \nonumber\\
&\geq \delta/2- O_k\bigbrac{\exp (-(\log \log N)^{c_k})}.
\end{align}
Denoting $\widetilde{\nu}_1=\tfrac{\nu_1+1}{2}$, we have $f,g\ll \widetilde{\nu}_1$ pointwise, and $\widetilde{\nu}_1$ satisfies the $(D_k,D_k,D_k)$-linear forms condition at scale $N$ with error $O_k((\log N)^{-c_k'})$. Hence, Lemma~\ref{neumann} together with the uniformity bound~\eqref{uni-bd} gives
\begin{align}\label{6.3}
\begin{split}
&\quad\frac{W}{N-b}\sum_{n\in \Z} \E_{d \in [\frac{N-b}{W}]} f(n)\cdots f(n+(k-1)d)\\
&\geq \frac{W}{N-b}\sum_{n\in \Z} \E_{d \in [\frac{N-b}{W}]} f\cdot 1_\Omega(n)\cdots f\cdot 1_\Omega(n+(k-1)d)\\
&= \frac{W}{N-b}\sum_{n\in \Z} \E_{d \in [\frac{N-b}{W}]} g\cdot 1_\Omega (n)\cdots g\cdot 1_\Omega (n+(k-1)d)
+O_k((\log N)^{-\gamma_k}) \\
&\quad + O_k\bigbrac{\norm{(f-g) \cdot1_\Omega}_{U^{k-1}[\frac{N-b}{W}]}}\\
&= \frac{W}{N-b}\sum_{n\in \Z} \E_{d \in [\frac{N-b}{W}]} g\cdot 1_\Omega (n)\cdots g\cdot 1_\Omega (n+(k-1)d)
+O_k\bigbrac{\exp\bigbrac{-(\log \log N)^{c_k}}},
\end{split}
\end{align}
where for the last step we used properties (1) and (2).

Set $M\coloneqq \frac{N-b}{W}$ and define $h\colon [kM]\to[0,1]$ by
\[
h(n)\coloneqq
\begin{cases}
g(n)1_\Omega(n)/2, & 1\leq n\leq M,\\
0, & M<n\leq kM.
\end{cases}
\]
By~\eqref{density-bd} and the assumption $\delta\geq \exp(-(\log\log N)^{c'})$, we may assume (for $N$ sufficiently large depending on $k$) that $\E_{n\in [M]} h(n)\geq \delta/4$. Applying Lemma~\ref{1bd-counting} (with $N$ replaced by $M$ and with $g$ replaced by $h$), we obtain
\begin{align*}
\frac{W}{N-b}\sum_{n\in \Z} \E_{d \in [M]} g\cdot 1_{\Omega}(n)\cdots g\cdot 1_{\Omega}(n+(k-1)d)
&\gg_k \E_{n,d\in [M]} h(n)\cdots h(n+(k-1)d)\\
&\gg_k \exp\bigl(-\exp((\log(1/\delta))^{C_k})\bigr),
\end{align*}
and in the case $k=4$,
\[
\frac{W}{N-b}\sum_{n\in \Z} \E_{d \in [M]} g\cdot 1_{\Omega}(n) g\cdot 1_{\Omega}(n+d) g\cdot 1_{\Omega}(n+2d) g\cdot 1_{\Omega}(n+3d)
\gg \exp(-\delta^{-C_4}).
\]
Inserting these estimates into~\eqref{6.3} completes the proof.
\end{proof}

\begin{proof}[Proof of Theorem~\ref{density-bound}]
If $N$ is large enough, we have
\[
\E_{n\in [N]} 1_\mathcal A(n) \Lambda (n) \geq \delta/3,
\]
since $\log n\geq (\log N)/2$ for $n\in [N^{1/2},N]$.
Setting $W=\prod_{p\leq (\log N)^{c_0}}p$, by the pigeonhole principle there exists some integer $1\leq b\leq W$ with $(b,W)=1$ such that
\[
\sum_{\substack{n\leq N \\ n\equiv b \!\!\!\pmod W }}1_\mathcal A(n) \Lambda (n) \geq \frac{\delta N}{4\phi(W)}.
\]
Fix this residue class $b$, and define $f(n) =\frac{\phi(W)}{W}\, 1_\mathcal A(Wn+b)\, \Lambda (Wn+b)$. Since $\mathcal{A}\subseteq [N]$, we have $\supp(f) \subseteq[\frac{N-b}{W}]$ and $\E_{n\in [\frac{N-b}{W}]}f(n) \geq \delta/3$. Thus, from Theorem~\ref{density-bd-2} we see that when $k\geq 5$ and $\delta\gg \exp (-(\log \log\log N)^{c_k})$ we have
\[
\frac{W}{N-b}\sum_{n\in \Z} \E_{d \in [\frac{N-b}{W}]} f(n)\cdots f(n+(k-1)d) \gg_k \exp\bigl(-\exp((\log(1/\delta))^{C_k})\bigr);
\]
while when $k=4$ and $\delta\gg (\log \log N)^{-c}$ we have
\[
\frac{W}{N-b}\sum_{n\in \Z} \E_{d \in [\frac{N-b}{W}]} f(n)f(n+d)f(n+2d)f(n+3d) \gg \exp(-\delta^{-C_4}).
\]
This completes the proof of Theorem~\ref{density-bound}.
\end{proof}

\appendix
\section{Type I and Type II Estimates for Nilsequences}

The ideas in this appendix are closely related to \cite[Section 5]{Leng}. However, we require Type I and Type II estimates for nilsequences along sparse arithmetic progressions, in which the common difference is larger than the complexity of the nilsequence. Since we have not been able to locate such estimates in the literature, we provide a detailed treatment here. 

\begin{lemma}\label{type-i-ii}
Let $A$ be a large absolute constant. Let $0<\delta<1$ be a parameter and let $1\leq b\leq W$ be integers with $(b,W)=1$. Suppose that $\psi\colon \Z\to\C$ is a 1-bounded function and that
\[
\Bigabs{\twosum{n\in [N]}{n\equiv b \mod W} h(n) \psi(\frac{n-b}{W})}\geq \frac{\delta N}{W}.
\]	
\begin{enumerate}
	\item If $h$ is a Type I sum of level $N^{2/3}$ at scale $N$, then there exists  $L\leq N^{2/3}$ such that for at least $\delta(\log N)^{-A} L$ elements $l\in [L/2,L]$ with $(l,W)=1$  we have
	\[
	\Bigabs{\twosum{m\leq N/L}{m\equiv b l^{-1}\mod W} \psi(\frac{ml-b}{W})} \geq \frac{\delta N}{WL}(\log N)^{-A},
	\]
	where $l^{-1}$ denotes the inverse of $l$ modulo $W$.
	\item If $h$ is a Type II sum of range $[N^{1/3}, N^{2/3}]$ at scale $N$, then there exist $N^{1/3} \leq L\leq N^{2/3}$ and $M\in [\delta N/L, N/L]$ such that
	\[
	\frac{\delta^{O(1)} N^2}{W^3}(\log N)^{-A} \ll  \sum_{r \mod W}\nolimits^{*} \threesum{M<m_i\leq 2 M}{m_i\equiv b r^{-1} \mod W}{i=1,2} \biggabs{\twosum{L<l\leq 2 L}{l \equiv r \mod W} \psi(\frac{m_1l-b}{W})\overline{\psi(\frac{m_2l-b}{W})}}^2.
		\]
\end{enumerate}
\end{lemma}

\begin{proof}
We only prove the second statement as the first is standard and simpler. Since $h$ is a Type II sum of range $[N^{1/3},N^{2/3}]$, we may write
\[
\frac{\delta N}{W} \leq \biggabs{\threesum{r,r'\mod W}{rr'\equiv b\mod W}{(rr',W)=1} \twosum{N^{1/3}\leq l\leq N^{2/3}}{l\equiv r \mod W} a_l \twosum{m\leq N/l}{m\equiv r'\mod W} b_m\psi(\frac{ml-b}{W})},
\]
where $a_l$ and $b_m$ are divisor-bounded functions.  Decomposing $[N^{1/3}, N^{2/3}]$ into dyadic intervals, it follows from the pigeonhole principle that there exist some $N^{1/3} \leq L\leq N^{2/3}$ and $\delta N\leq ML\leq N$ such that
\[
\frac{\delta^2 N}{W}(\log N)^{-2} \ll \threesum{r,r'\mod W}{rr'\equiv b\mod W}{(rr',W)=1} \twosum{L<l\leq 2 L}{l\equiv r \mod W} |a_l| \biggabs{\twosum{M<m\leq 2 M}{m\equiv r'\mod W} b_m\psi(\frac{ml-b}{W})}.
\]
After Cauchy-Schwarz inequality, we find
\begin{multline*}
	\frac{\delta^{4} N^2}{W^2}(\log N)^{-4} \ll \threesum{r,r'\mod W}{rr'\equiv b\mod W}{(rr',W)=1} \twosum{L<l\leq 2 L}{l\equiv r\mod W}|a_l|^2\\
	\times \threesum{r,r'\mod W}{rr'\equiv b\mod W}{(rr',W)=1} \twosum{L<l\leq 2 L}{l\equiv r\mod W}\biggabs{\twosum{M<m\leq 2 M}{m\equiv r'\mod W} b_m\psi(\frac{ml-b}{W})}^2.
\end{multline*}

Using standard average bounds for the divisor function and noting that for any $r$ coprime to $W$ there is exactly one $r'\mod W$ such that $rr'\equiv b\mod W$, we obtain
\[
	\frac{\delta^{4} N^2}{LW^2}(\log N)^{-O(1)}\ll \threesum{r,r'\mod W}{rr'\equiv b\mod W}{(rr',W)=1} \twosum{L<l\leq 2 L}{l\equiv r\mod W}\twosum{M<m_1,m_2\leq 2 M}{m_i\equiv r'\mod W} b_{m_1}\bar{b_{m_2}} \psi(\frac{m_1l-b}{W}) \overline{\psi(\frac{m_2l-b}{W})}.
\]
Since $b_{m_i}$ is divisor-bounded, Shiu's bound gives
\[
\twosum{M<m_i\leq 2 M}{m_i\equiv r' \mod W}|b_{m_i}|^2 \ll \frac{M (\log M)^{O(1)}}{\phi(W)}.
\]	
Changing the order of summation and applying  Cauchy-Schwarz inequality once more yields 
\[
	\frac{\delta^{O(1)} N^2}{W^3}(\log N)^{-O(1)} \ll  \sum_{r \mod W}\nolimits^{*} \threesum{M<m_i\leq 2 M}{m_i\equiv b r^{-1} \mod W}{i=1,2} \biggabs{\twosum{L<l\leq 2 L}{l \equiv r \mod W} \psi(\frac{m_1l-b}{W})\overline{\psi(\frac{m_2l-b}{W})}}^2,
\]
which is the desired estimate.
\end{proof}

\begin{lemma}[Type I estimate]\label{type-i}
Let $1\leq b\leq W\leq N^{1/10}$ with $(b,W)=1$, let $(\log N)^{A}\leq 1/\delta \leq \exp((\log N)^{1/10^{10}})$ for some absolute constant $A>0$. Let $G/\Gamma$ be a nilmanifold of step $s$, degree $k$, dimension $d$ and complexity $K$, with one-dimensional vertical component. Let $F\colon G/\Gamma \to\C$ be a 1-Lipschitz function and a vertical character with frequency $\xi$ such that $|\xi|\leq K/\delta$. Suppose that $L\leq N^{2/3}$ and there are $\gg \delta^{O(1)}L$ elements $l\in [L/2,L]$ such that $(l,W)=1$ and
\[
\Biggabs{\twosum{m\leq N/L}{m\equiv b l^{-1} \mod W} F(g(\frac{ml-b}{W})\Gamma)}\geq \frac{\delta N}{WL}.
\]
Then $g$ admits a factorisation
\[
g(n)=\eps (n)\,g'(n)\,\gamma(n),
\]
where $g'$ takes values in a $(K/\delta)^{O_k(d^{O_k(1)})}$-rational subgroup of $G$ of strictly smaller step than $G$;
$\eps $ is $\bigl((K/\delta)^{O_k(d^{O_k(1)})},\,N/W\bigr)$-smooth;
and $\gamma$ is $(K/\delta)^{O_k(d^{O_k(1)})}$-rational.
\end{lemma}

\begin{proof}
Make the change of variables $m = Wn + b m_W$, where $m_W$ is the unique solution to $lm_W\equiv 1\pmod W$ with $bm_W\in [W]$.
 Then $n\le N/(WL)$, and for $\gg \delta^{O(1)}L$ integers $l\in [L/2,L]$ with $(l,W)=1$ we have
\[
\Bigabs{ \sum_{n\in  [\frac{N}{WL}]} F(g(ln+\tfrac{bm_Wl-b}{W})\Gamma)}\geq \frac{\delta N}{WL}.
\]	

Set $\tilde g_l=g(l\cdot+\tfrac{bm_Wl-b}{W})$ and $g_l= g(l\cdot)$; then both are polynomial sequences in $\poly(\Z,G_\bullet)$,  and the preceding bound becomes $|\sum_{n\in [\frac{N}{WL}]} F(\tilde g_l(n)\Gamma)|\geq \frac{\delta N}{WL}$. One may apply \cite[Theorem 4]{Leng} to find linearly independent horizontal characters $\eta_1,\dots,\eta_t$ of size at most $(K/\delta)^{O_k(d^{O_k(1)})}$ such that
\[
\|\eta_i \circ \tilde g_l\|_{C^\infty[\frac{N}{WL}]}\ll (K/\delta)^{O_k(d^{O_k(1)})}
\]
 holds for $\gg (\delta/K)^{O_k(d^{O_k(1)})}L$ many $l\in [L/2,L]$. After scaling the $\eta_i$ appropriately, it follows from \cite[Lemma A.12]{Leng} that
 \[
 \|\eta_i \circ  g_l\|_{C^\infty[\frac{N}{WL}]}\ll (K/\delta)^{O_k(d^{O_k(1)})}
 \]
for $\gg (\delta/K)^{O_k(d^{O_k(1)})}L$ many $l\in [L/2,L]$. One then deduces from \cite[Lemmas A.10--A.11]{Leng} that there exists an integer $q\ll (K/\delta)^{O_k(d^{O_k(1)})}$ such that
\[
\|q\eta_i \circ g\|_{C^\infty[N/W]}\ll (K/\delta)^{O_k(d^{O_k(1)})}.
\]
Renaming $q\eta_i$ as $\eta_i$, we may now invoke \cite[Lemma A.1]{Leng} to obtain a factorisation 
\[
g=\eps g'\gamma,
\]
where $g'\in \poly(\Z,G^*)$ with $G^*=\set{g\in G\colon \eta_i(g)=0 \text{ for all }i}$ and $G_i^*=G_i\cap G^*$. The subgroup $G^*$ is $(K/\delta)^{O_k(d^{O_k(1)})}$-rational and has step strictly smaller than $s$. Moreover, $\eps $ is $\bigl((K/\delta)^{O_k(d^{O_k(1)})},\,N/W\bigr)$-smooth and $\gamma$ is $(K/\delta)^{O_k(d^{O_k(1)})}$-rational, as required.

\end{proof}

\begin{lemma}[Type II estimate]\label{type-ii}
Let $0<\delta<1/2$ and let $1\leq b\leq W\leq N^{1/10}$ with $(b,W)=1$. Let $G/\Gamma$, $F\colon G/\Gamma\to\C$, and the vertical frequency $\xi$ be as in Lemma~\ref{type-i}. Suppose that $N^{1/3} \leq L\leq N^{2/3}$ and $\delta N\leq ML\leq N$ and that we have
\[
	\frac{\delta N^2}{W^3} \leq  \sum_{r \mod W}\nolimits^{*} \threesum{M<m_i\leq 2 M}{m_i\equiv b r^{-1} \mod W}{i=1,2} \biggabs{\twosum{L<l\leq 2 L}{l \equiv r \mod W} F \bigbrac{g(\tfrac{m_1l-b}{W})\Gamma} \overline{F \bigbrac{g(\tfrac{m_2l-b}{W})\Gamma}}}^2.
\]
Then there exists a factorisation
\[
g(n)=\eps (n)\,g'(n)\,\gamma(n),
\]
such that $g'$ takes values in a $(K/\delta)^{O_k(d^{O_k(1)})}$-rational subgroup of $G$ of strictly smaller step than $G$;
$\eps $ is $\bigl((K/\delta)^{O_k(d^{O_k(1)})},\,N/W\bigr)$-smooth;
and $\gamma$ is $(K/\delta)^{O_k(d^{O_k(1)})}$-rational.
	
\end{lemma}

\begin{proof}
By pigeonhole principle there exist $\gg \delta^{O(1)}W$ residues $r\in [W]$ with $(r,W)=1$, and for each such $r$ there are $\gg \delta^{O(1)}(M/W)^2$ pairs $(m_1,m_2)\in [M,2M]^2$ with $m_1\equiv m_2\equiv br^{-1} \mod W$ such that
\[
\delta^{O(1)}L/W \ll \Bigabs{\twosum{L<l\leq 2 L}{l\equiv r \mod W} F \bigbrac{g(\tfrac{m_1l-b}{W})\Gamma} \overline{F \bigbrac{g(\tfrac{m_2l-b}{W})\Gamma}}}.
\]
Write $l=r+Wn$, so that $L/W<n\leq 2 L/W$, and the preceding inequality becomes
\[
\delta^{O(1)}L/W \ll \Bigabs{ \sum_{L/W<n\leq 2 L/W}  F \bigbrac{g(m_1n+\tfrac{m_1r-b}{W})\Gamma} 
\overline{ F \bigbrac{g(m_2n+\tfrac{m_2r-b}{W})\Gamma} } }
\]
for $\gg \delta^{O(1)}W$ choices $r$ and $\gg \delta^{O(1)}(M/W)^2$ such pairs $(m_1,m_2)$. 

For each such triple $(m_1,m_2, r)$, define the polynomial sequence
\[
\tilde g_{m_1,m_2,r} =\bigbrac{ g(m_1\cdot + \tfrac{m_1r-b}{W}), g(m_2\cdot +\tfrac{m_2r-b}{W})},
\]
and  for such a pair $(m_1,m_2)$ define $g_{m_1,m_2} =\bigbrac{g(m_1\cdot),g(m_2\cdot)}$.
Then  $\tilde g_{m_1,m_2,r}, g_{m_1,m_2} \in\poly(\Z,G^2)$. Applying \cite[Theorem 4]{Leng} to $\tilde g_{m_1,m_2,r}$, we obtain (for each such triple) linearly independent horizontal characters $\eta_1,\dots,\eta_t$ on $G\times G$ of size at most $(K/\delta)^{O_k(d^{O_k(1)})}$ and such that
\[
\|\eta_i \circ \tilde g_{m_1,m_2,r}\|_{C^\infty [L/W]} \ll (K/\delta)^{O_k(d^{O_k(1)})}.
\]
By pigeonhole principle, using \cite[Lemma A.12]{Leng} and scaling up $\eta_i$ if necessary, we may assume that
\[
\|\eta_i \circ  g_{m_1,m_2}\|_{C^\infty [L/W]} \ll (K/\delta)^{O_k(d^{O_k(1)})}
\]
holds for all $\eta_1,\dots,\eta_t$ and  for at least $(\delta/K)^{O_k(d^{O_k(1)})}M^2/W$ pairs $(m_1,m_2)\in [M,2M]^2$ with $m_1\equiv m_2 \mod W$.

Write each $\eta_i$ as $\eta_i=\eta_i^{(1)}\oplus \eta_i^{(2)}$, where $\eta_i^{(1)}, \eta_i^{(2)}\colon G\to \T$ are horizontal characters,  not both zero. Write
 \[
\eta_i^{(1)}\circ g(n)=\sum_{j=0}^k \beta_{j,i}\,n^j,
\qquad
\eta_i^{(2)}\circ g(n)=\sum_{j=0}^k \beta'_{j,i}\,n^j.
\]
Then
\[
\eta_i\circ g_{m_1,m_2}(n)
=
\sum_{j=0}^k \bigl(\beta_{j,i}m_1^j+\beta'_{j,i}m_2^j\bigr)n^j.
\]
By \cite[Lemma A.10]{Leng}, there exists an integer $q\ll_k 1$ such that for all $1\le i\le t$ and $1\le j\le k$,
\[
\| q(\beta_{j,i}m_1^j+\beta_{j,i}'m_2^j)\|_\T \ll (\tfrac{L}{W})^{-j} \| \eta_i \circ g_{m_1,m_2}\|_{C^\infty[L/W]}\ll (K/\delta)^{O_k(d^{O_k(1)})}(W/L)^j
\]	
 for at least $(\delta/K)^{O_k(d^{O_k(1)})}M^2/W$ pairs $(m_1,m_2)\in [M,2M]^2$ with $m_1\equiv m_2 \mod W$. Fix $i$ and choose an index with $\eta_i^{(1)}\neq 0$ (the case $\eta_i^{(2)}\neq 0$ is analogous). 
 By pigeonholing in $m_2$,
we may select some $m_2\in[M,2M]$ that occurs in $\gg (\delta/K)^{O_k(d^{O_k(1)})}M$ admissible pairs $(m_1,m_2)$, noting that there are only $M/W$ choices of $m_2$ under the assumption $m_1\equiv m_2 \mod W$.
Subtracting the relation for two such $m_1$'s  yields
\[
\|q\,\beta_{j,i}(m_1^j-m_1'^j)\|_{\T}
\ll 
(K/\delta)^{O_k(d^{O_k(1)})}\Bigl(\frac{W}{L}\Bigr)^j,
\]
and hence one obtains
\[
\|qm^j\beta_{j,i}\|_{\T} \ll (K/\delta)^{O_k(d^{O_k(1)})}(W/L)^j
\]
 for $\gg (\delta/K)^{O_k(d^{O_k(1)})}M$ integers $m\in [-M,M]$. A standard Waring-type argument then gives an integer $q'\ll (K/\delta)^{O_k(d^{O_k(1)})}$ such that
 \[
\|q'\beta_{j,i}\|_{\T}
\ \ll\
(K/\delta)^{O_k(d^{O_k(1)})}\Bigl(\frac{W}{ML}\Bigr)^j
\qquad (j=1,\dots,k).
\]
By \cite[Definition 2.2]{Leng} and the assumption $\delta N\leq ML\leq N$, this implies
\[
\|q'\eta_i^{(1)}\circ g\|_{C^\infty[N/W]}
\ \ll\ 
(K/\delta)^{O_k(d^{O_k(1)})}.
\]
 The desired factorisation now follows from~\cite[Lemma~A.1]{Leng} for some polynomial sequence $g'$ taking values in $\tilde{G}\coloneqq \bigcap_{j=1}^t \textnormal{ker}(\eta_j)$.
It remains to justify the step drop. 

Recall that $\eta_i=\eta_i^{(1)}\oplus \eta_i^{(2)}$ as above, so that
$\eta_i(x,y)=\eta_i^{(1)}(x)+\eta_i^{(2)}(y)$.
Let
\[
\widetilde G_1 \coloneqq \{(x,y)\in G\times G:\ \eta_i(x,y)=0\ \text{for all }1\le i\le t\}.
\]
By the conclusion of \cite[Theorem 4]{Leng} (as used in the proof of \cite[Proposition~5.4]{Leng}),
$\widetilde G_1$ is a rational subgroup of $G\times G$ of step at most $s-1$, where $s$ is the step of $G$.
Note that
\[
\widetilde G = \{h\in G:\ \eta_i^{(1)}(h)=\eta_i^{(2)}(h)=0\ \text{for all }1\le i\le t\}.
\]
Then for any $h\in \widetilde G$ and any $i$ we have
$\eta_i(h,\mathrm{id})=\eta_i^{(1)}(h)+\eta_i^{(2)}(\mathrm{id})=0$,
so $(h,\mathrm{id})\in \widetilde G_1$; hence $\widetilde G\times\{\mathrm{id}\}\subseteq \widetilde G_1$.
Since subgroups of an $(s-1)$-step nilpotent group are again $(s-1)$-step, it follows that
$\widetilde G\times\{\mathrm{id}\}$ has step $\le s-1$, and therefore (via the embedding $h\mapsto (h,\mathrm{id})$)
$\widetilde G$ itself has step $\le s-1$, i.e. strictly smaller than the step of $G$.
This is the required step reduction, which completes the proof.
\end{proof}

\bibliographystyle{plain}

\bibliography{md_references.bib}

\end{document}